\documentclass[12pt,twoside, a4paper, english, reqno]{amsart} %%,draft
\usepackage[foot]{amsaddr}%%
\usepackage{a4wide}
\usepackage{amscd}
\usepackage{amssymb}
\usepackage{amsthm}
\usepackage{graphicx}
\usepackage{amsmath,calrsfs}
\usepackage{latexsym}
\usepackage{enumitem,dsfont}

\usepackage{color,graphics}

\usepackage[usenames,dvipsnames]{xcolor}
\usepackage{pdfsync}
\usepackage{ulem,cancel,pgf}
\usepackage[pdfdisplaydoctitle,colorlinks,breaklinks,urlcolor=blue,linkcolor=red,citecolor=blue]{hyperref} 
\usepackage[nameinlink,capitalise]{cleveref}%must always be the last

\usepackage{frcursive}
\theoremstyle{plain}
\newtheorem{theorem}{Theorem}[section]
\theoremstyle{plain}
\newtheorem{lemma}[theorem]{Lemma}
\newtheorem{prop}[theorem]{Proposition}
\newtheorem{cor}[theorem]{Corollary}
\newtheorem{ex}{Example}[section]
\theoremstyle{definition}
\newtheorem{definition}{Definition}[section]

\newtheorem{remark}{Remark}[section]

\newtheorem*{maintheorem*}{Main Theorem}
\newtheorem*{maincorollary*}{Main Corollary}

\DeclareFontFamily{U}{BOONDOX-calo}{\skewchar\font=50 }
\DeclareFontShape{U}{BOONDOX-calo}{m}{n}{
	<-> s*[1.05] BOONDOX-r-calo}{}
\DeclareFontShape{U}{BOONDOX-calo}{b}{n}{
	<-> s*[1.05] BOONDOX-b-calo}{}
\DeclareMathAlphabet{\mathcalb}{U}{BOONDOX-calo}{m}{n}
\SetMathAlphabet{\mathcalb}{bold}{U}{BOONDOX-calo}{b}{n}
\DeclareMathAlphabet{\mathbcalb}{U}{BOONDOX-calo}{b}{n}
\newcommand{\R}{\ensuremath{\mathbb{R}}}

\newcommand{\E}{\ensuremath{\mathbb{E}}}

\newcommand{\eps}{\ensuremath{\varepsilon}}

\def\V{V}

\def\Vprime{V^\prime}

\numberwithin{equation}{section} \allowdisplaybreaks

\title[	Ergodicity for	 stochastic	obstacle	problems]
{Ergodicity for stochastic T-monotone	parabolic	obstacle	problems }

\author[Yassine Tahraoui]{ Yassine Tahraoui}
\address{Scuola Normale Superiore, Piazza dei Cavalieri, 7, 56126 Pisa, Italy}
\email{\href{mailto:yassine.tahraoui at sns.it}{yassine.tahraoui at sns.it}}
\date{\today }
\thanks{This work is funded by national funds through the FCT - Funda\c c\~ao para a Ci\^encia e a Tecnologia, I.P., under the scope of the projects UIDB/00297/2020 and UIDP/00297/2020 (Center for Mathematics and Applications).
}

\usepackage{comment}

\begin{document}

	\begin{abstract}
This work aims to investigate	the 		existence	of		ergodic	invariant	measures	and	its	uniqueness,	associated	with	obstacle	problems	governed	by	a	T-monotone	operator defined	on	Sobolev	spaces	and	driven	by	a			multiplicative	noise	in		a	bounded	domain	of	$\mathbb{R}^d$ with homogeneous boundary conditions.	We	show	that	the	solution	 defines a Markov-Feller 	semigroup defined on  the space of real bounded continuous   functions of 	a convex subset  related to the obstacle and 	we	prove	the	existence	of		ergodic	invariant	measures	and	its	uniqueness,	under suitable
assumptions.		Our method  relies on a combination of 	"Krylov-Bogoliubov theorem", "Krein–Milman theorem" and  Lewy-Stampacchia inequalities to	control	the reflection measure.
  		\end{abstract} 
  %%	to	construct	an	invariant	measure,
  	\maketitle

%\tableofcontents              

\section{Introduction}
We are interested in		the	existence	of	ergodic	invariant	measures 	and	its	uniqueness,			associated	with	the	solution $(u,k)$  to 	the following obstacle problem
\begin{align}\label{obstacle-pb}
\begin{cases}
&du+[A(u)+\mathbf{F}(u)+ k ]ds= fds+ G(u)dW \quad   \\ 
&  \langle k,u-\psi\rangle_{V^\prime,V}  = 0 \hspace*{0.2cm} \text{and} \hspace*{0.2cm} \langle -k,\varphi\rangle_{V^\prime,V}\geq 0 \quad \forall \varphi \in	V, \varphi\geq 0   \\ 
&u \geq \psi , \hspace*{0.2cm} u(0,\cdot)=u_0 \geq \psi ,
\end{cases}
\end{align}
where  $A$ is a nonlinear T-monotone operator defined on a Banach	space $V$ with values in $V^\prime$, where  $V^\prime$ denotes the  dual	space of $V$, $\mathbf{F}$ is  a Lipschitz Nemytskii operator, $f$	is	an		external	force, $G(\cdot)$ defines a Hilbert–Schmidt operator  and  $(W(t))_{t\geq0}$ is a $Q$-Wiener process with values in  a separable Hilbert space $H$ defined on a filtered probability space $(\Omega,\mathcal{F}, (\mathcal{F}_t)_{t\geq 0}, P).$  The \textit{reflection measure} $k$ prevents the solution $u$ from crossing the obstacle $\psi$. The second line in \eqref{obstacle-pb} (we will refer to it as "minimality condition") says that $k$  is pushing away  $u$ to cross $\psi$ in minimal  sense \textit{i.e.}  the support of $k$ is included in the contact set $\{u=\psi\}$ while $k$ should equals $0$ in the free set $\{u>\psi\}$ and the free boundary $\partial\{u=\psi\}=\partial\{u>\psi\}$ is  part of the unknowns.\\
 
%%
%\begin{align*}
%\begin{cases}
%du+A(u)ds+ k ds= fds+ G(u)dW \quad & \text{in} \quad D\times  \Omega_T, \\ 
%u(t=0)=u_0 & \text{in} \quad D,\\ 
%u \geq \psi & \text{in} \quad  D\times  \Omega_T, \\  
%u=  0   &\text{on} \quad\partial D\times \Omega_T,  \\
%\langle k,u-\psi\rangle_{V^\prime,V}  = 0 \hspace*{0.2cm} \text{and} \hspace*{0.2cm} -k \in	(V^\prime)^+  & \text{a.e.	in} \quad  \Omega_T,
%\end{cases}
%\end{align*}
%%
\subsubsection{Motivations and state of the art}
Obstacle  problems 	of the form \eqref{obstacle-pb}   appear in the mathematical modeling of  several  phenomena		 including	Physics, Biology and Finance.	For example, 	 the evolution of damage in a continuous	medium taking into account the microscopic description \cite{Bauz},	the	 American option	pricing,	some	questions	of	flows in porous media, phase transition	and	some	statistical
mechanic	problems,	see \textit{e.g.}	\cite{Otobe-2004,Yas}	and	references	theirin.	It's worth to mention that a formulation by using variational inequalities is largely used in the deterministic setting, we refer to \textit{e.g.} \cite{GMYV,GYV,	Mok-Yas-Val} and their references for the interested reader. \\

Concerning finite dimensional  SDE (stochastic differential equation) with reflection, without exhaustiveness,  we refer to 
  \cite{Skorohod61}, where Skorohod  constructed the solution of  stochastic differential equations on $\mathbb{R}^+$
with reflection at $0$ by using the properties of   the space of real continuous functions. 
Then,  the  existence and uniqueness have been proven  in \cite{Cepa98} for SDEs in the presence of a multivalued maximal monotone operator with normal reflection associated with a closed convex subsets  of $\mathbb{R}^d$. In  \cite{Barbu-Daprato-08-1},  the authors   studied 
 the generator of the transition semigroup of SDEs   with boundary reflection on a closed and convex subset of $\mathbb{R}^d$.
The above works concern a finite-dimensional  SDEs.
In infinte dimensional setting, without	seeking to be exhaustive, let us mention  
\cite{Nulart-Pardoux92} for  the existence and uniqueness result about reflected solutions of the heat equation on the spatial interval $[0,1]$ with Dirichlet boundary conditions, driven by an additive space-time white noise. It is worth mentioning that this can be considered  as an extension  of one-dimensional SDEs reflected at $0$. Their  approach relies on some results about deterministic variational inequalities. Then, existence of     reflected solutions of the heat equation on the spatial interval $[0,1]$ with Dirichlet boundary conditions, driven by multiplicative noise, has been proved in  \cite{DonatiMartinPardoux}. For the formulation as SVI (stochastic variational inequalities) in 1D, see \cite{Hau-Par}. It seems that the 1D restriction is related to  the fact that the expression of the minimality condition requires an appropriate regularity, which is only valid in 1D due to Sobolev embedding theorem. Thus, in a higher dimension, either more in-depth analysis  of the regularity of the solution and of the reflection measure is needed to express the minimality condition, see \textit{e.g.} \cite{NYGA-additive} based on parabolic potential theory or some  duality arguments, see \textit{e.g.} \cite{NYGA,YV22} and their references.\\

%%Another advantage of keeping the  system of equations is that one can use the regularity result available for the equations-----
  %Since the regularity of  reflection measure should be related to the rlike considering only additive noiseegularity of the obstacle,  \textit{via}
  
It is worth mentioning that the formulation by using either the  SVI or  multi-valued SPDEs
  may lead to restrictions  in order to  transform the problem into random SVI or weakening the notion of solution. We prefer to work with the coupled system \eqref{obstacle-pb}, which allows to  use directly the infinite-dimensional stochastic calculus, namely It\^o's formula, as far as we have some control over the reflection measure, 
 thanks to  Lewy-Stampacchia's inequalities, see  \cite{YV22}. In addition, the rigorous formulation of stochastic  obstacle problems as SVI or  differential inclusion of measure is not always true because it can suffer from strong singularities close to the contact set.\\
 
 %We refer  \textit{e.g.} to \cite{NYGA-additive}  for a recent result about  the existence of solutions to an  obstacle problem with additive noise.\\

Concerning  stochastic	PDEs with obstacle in different dimensions,  let us refer to 	 \cite{DenisMatoussiZhang} 	 for
 an existence and uniqueness result for quasi-linear stochastic PDEs with obstacle, by using  the parabolic potential theory. 
In \cite{Zhang07-mon-multi}, the author  proved the existence and uniqueness of solutions to multivalued stochastic evolution equations involving maximal monotone operators in some Hilbert space and single valued  strongly monotone operator.  
Besides	the	well	posedness,	a	regularity	properties	of	the	reflected	stochastic	measure		was	studied	in	\cite{YV22},	namely stochastic Lewy-Stampacchia’s inequalities in	the	case	of	stochastic  T-monotone obstacle problems	and	then	\cite{IYG,NYGA}	for	 stochastic scalar conservation laws with constraint	and		pseudomonotone	parabolic	obstacle	problems in the presence of a multiplicative noise.  Recently,   an existence result based on a comparison theorem and parabolic capacity  has been proved in  \cite{NYGA-additive}, about the solutions to  nonlinear, pseudomonotone, stochastic diffusion-convection evolution problem in the  presence of  additive noise on a bounded  domain, with homogeneous boundary conditions and reflection.
Finally, the large deviations for stochastic obstacle problems were studied	\textit{e.g.}	in	\cite{Matou21,Yassine23,Zhang12}.\\

Concerning	the		ergodicity,	a	relation	between	the	"spatial	and	temporal	averages"	for	a	large	class	of	dynamics,			the		existence	of	invariant	measures	is	an	important	factor	to	investigate	the		ergodicity.	Moreover,	it concerns an asymptotic of the dynamic for a large time.		There exists many works about the invariant measures and ergodicity for finite	and	infinite dimensional dynamical systems.	First, let us mention the pioneering work	\cite{Kryloff-Bogoliouboff},	where	the	authors	established	a	procedure	to	construct	an	invariant	measure	for	a	dynamical	system.	Concerning some results about the invariant measures and ergodicity for stochastic evolution	equations, without exhaustiveness, let us refer to	\cite{Daprato-lecture,Daprato,Daprato-ergodicity,Flandoli,Peszat-Zabczyk,Prevot-Rockner,RoK15,Yassine-Fernanda}	and	their	references. We refer also to \cite{Barbu-Daprato-10,Gess} and	their	references on the question of ergodicity for the singular stochastic $p$-Laplace operator in the presence of an additive noise and \cite{Gess2} for stochastic evolution
inclusions in Hilbert spaces.\\
%%%,	which	known	in	the	literature	as	"Krylov-Bogoliubov	procedure".
%%%%%%%%%%
%%%%\textcolor{red}{ the unique existence of solutions could only be shown for additive noise and under
%strong dimensional restrictions. The principal idea of most of these works is the concept of (stochastic) evolutionvariational inequalities (EVI), 3 thus weakening the notion of solutions to (1.3). However, the approach via EVI has multiple drawbacks. First, it relies on the transformation of (1.3) into a random PDE and hence is restricted to simple structures of noise, such as additive or linear multiplicative noise. Second, due to the weaker notion of solutions it is hard to prove uniqueness. In fact, so far uniqueness of EVI could only be proven in case of sufficiently regular additivenoise. Third, the construction of solutions to EVI still requires a %%%coercivity condition of the type\\}
%%%%%%%%%

	Concerning	the		invariant	measures	to	stochastic	obstacle	problems,	the	invariant	measure	to	the			heat	equation	with		an	additive	white-noise	in	1D,				with	reflection	at	$0$			were	studied	in	\cite{Otobe-2004,Zambotti}.	In		\cite{Zambotti},		the	existence	of	an	explicit	symmetrizing	invariant	measure	on	$C([0,1])$		were	studied			and	then		the	relation	between			the stationary distribution and
	the Gibbs measure		on	$C(\mathbb{R})$	were	investigated	in	\cite{Otobe-2004}.		In	\cite{Zhang10},	a Harnack inequalities for the semigroup associated	with	the	stochastic	heat	equation		 with reflection	at	$0$	were	proved.
In	\cite{Kalsi-20},	the	author	proved		the	existence of invariant measures
for the	reflected	heat	equation	with	multiplicative	noise on	$[0,1]$. 
One	notices	that	all	these	results	are	in	1D	space	setting.
In an abstract Hilbertian framework,	the	authors		in \cite{Barbu-Daprato-08}	studied		the	well-posedeness	and		the	invariant	measures  of	a stochastic variational inequality on closed convex bounded subsets with nonempty interior and smooth boundary of a Hilbert space  in	the	presence	of	an	additive	noise. We refer also to  \cite{Zhang07-mon-multi} about
the existence and uniqueness of invariant measures associated with the solutions  of  a class of multivalued nonlinear stochastic partial differential equations  involving maximal monotone operators in some Hilbert space and single valued  strongly monotone operators.\\ 

		The	results	concern	the	invariant	measures	for		stochastic	obstacle	problems		are		associated	either		with	the	stochastic	heat	equations	in	1D,		a	linear	operators	in	the	frame	of	Hilbert	spaces	 or strongly monotone operator by using formulation based on  multivalued SPDEs  involving maximal monotone operators in some Hilbert spaces. On the other hand, it is natural to consider \eqref{obstacle-pb} in any space dimension in the presence of non linear differential operators, as a natural extension of the problem studied in \cite{DonatiMartinPardoux,Nulart-Pardoux92}. A typical example we have in mind is  $A(u)=-\Delta_pu=-\text{div}[\vert	\nabla	u\vert^{p-2}\nabla	u], 1<p<+\infty$  in \eqref{obstacle-pb}. The operator  $-\Delta_p$ is a natural generalization of the linear Laplace operator and it has interesting features. In particular, it   degenerates  if 	 $2<p<+\infty$ and it becomes singular when $1<p<2$. To	the	best	of	author's	knowledge, the existing papers do not ensure the existence of an ergodic invariant measure  if  $2<p<+\infty$. Moreover, there doesn’t exist in the literature a result		about  the  existence of invariant measure for stochastic obstacle problems with singular  $\Delta_p$.\\
    %%%%without  some additional restrictions to get uniqueness result for such kind of operator
%  \textcolor{red}{  and	there doesn’t exist in the literature a result		about	stochastic	obstacle	problems	governed	by		a	non-linear	T-monotone	operators,	for	example		$\Delta_pu=-\text{div}[\vert	\nabla	u\vert^{p-2}\nabla	u]$	instead	of a linear operator in any space dimension,	including		1D	setting.	
%Moreover,	the	only	result		in	abstract	setting	is	the	one	in	\cite{Barbu-Daprato-08},	where	the	main	operator	is	a	linear	one.	Furthermore, the majority of results are established in the presence of additive noise.	\\
%TO DELETE---\\}

Our	aim,	in	this	contribution,		is	to	propose	a		results	about				the	existence	and	uniqueness	of	ergodic	invariant	measure	associated	with	an	obstacle	problem	governed	by	a	T-monotone	operator	and	in	the	presence	of	a	multiplicative	noise,	based	on	the	well-posedness	and	Lewy-Stampacchia's	inequalities	proved	in	\cite{YV22}.	First, we check the well-posedness of \eqref{obstacle-pb} in the case $1<p<2$, based on  \cite{YV22}. Then, we study the existence and uniqueness of  ergodic invariant measure associated with the solution.  Our result includes the existence of ergodic invariant measure when  $A(u)=-\Delta_pu; 2\leq p<+\infty$ without any restrictions if $2< p<+\infty$ and the existence of invariant measure if  $A(u)=-\Delta_pu; \max(1,\frac{2d}{d+2})<p<2$.  In particular, it includes
	the	singular	$\Delta_p,1<p<2$	in	1D	and	2D.		 The uniqueness  holds    under some natural assumptions. We	recall	that	\eqref{obstacle-pb}	is	a	free	boundary	type	problem,		where	the	unknowns	are	$u$		and	the	 reflection measure (Lagrange	multiplier)	$k$,	see	Definition \ref{def-1}	for	the precise formulation and  notion of solution. Our arguments rely on a combination of the  unbounded penalization (non-Lipschitz in the case of singular operators \footnote{For example, the case  $A(u)=-\Delta_pu; 1<p<2$.}),  Lewy-Stampacchia inequalities  to control the reflection measure, Krylov-Bogoliubov method to construct an invariant measure and Krein–Milman theorem to show that the set of ergodic invariant measures is not empty.\\
    
       %%%% under some constraints on spatiadimension $\frac{2d}{d+2}<p<2$ 
       %%and with some restriction in higher dimension	$\frac{2d}{d+2}<p<2,	d\geq		3$,	where	$d$	is	the	space	dimension.
       %%%%In addition,	it	can	be	also	interpreted	as	a	coupled	system	of	SPDEs,
        It is worth mentioning that the management of	$k$\footnote{It is worth recalling that $k$ depends non-linearly and implicitly on the solution $u$.},		resulting	from	the	interaction	between	the	solution	and		the	obstacle	near	to	the	contact	set,	 is crucial	 in order to apply infinite dimensional It\^o formula to derive some estimates. In our approach, 
         $k$ is an adapted process with values in $V^\prime$ and not  in a some Hilbert space. We use	 	Lewy-Stampacchia's inequalities	 to control it.  Hence, we can use directly  the classical  infinite dimensional It\^o formula (see \textit{e.g.} \cite{RoK15}) and derive some estimates in the presence of external forces as well. Thus, using 
          the form   \eqref{obstacle-pb} and  Lewy-Stampacchia's inequalities allow to also consider the case of singular operators. On the other hand, we  prove that   the semigroup associated with solution of \eqref{obstacle-pb} is 
	bounded,	 Markov-Feller 	and	stochastically	continuous	in the space of real bounded continuous   functions on 		$K_\psi$\footnote{$K_\psi$ is a convex subset in a Hilbert space related to the obstacle $\psi$.},
see Lemma \ref{Markov-Lemma}.    
     The existence of an invariant measure follows by using the Krylov-Bogoliubov theorem, after showing the tightness of an average measure constructed using the law of the solution of \eqref{obstacle-pb}. Then we show that the set of invariant measures is tight to get the existence of an ergodic invariant measure by using Krein–Milman theorem, where   Lewy-Stampacchia's inequalities is used to control the reflection measure. Finally, the uniqueness of invariant measure is proved under some natural assumptions.

\subsubsection*{Organization	of	the	paper}
The article is organized as follows:			in	Section	\ref{Sec-formulation},	after giving 		the hypotheses,	we	recall		the	well-posedness	result	 about the	stochastic T-monotone	obstacle problems.	Then,	we	present	the	main	results	of	this	work.		Section	\ref{Sec-Proof}		is	devoted	to	the	proof	of	the	main	results	in	the	following	way:	first,	we	recall	the	method	of	construction	of	the	solution	\textit{via}	penalization,	which	serves	to	verify	that	the	semigroup	associated	with	the	solution	of		\eqref{I-invariant----}	defines	Feller-Markov	process.	Secondly,	we prove that the set of invariant measures is not empty and tight if $2\leq p<+\infty$.	Finally,	we	prove	Theorem	\ref{THm2}	and Theorem	\ref{THm2-2}  about	the	uniqueness	of	the	invariant	measure,	under	suitable	assumptions.					Section	\ref{sec-appendix}	concerns	the		well-posedness		of	a	T-monotone	obstacle	problem	in	the	presence	of		zero	order	Lipschitz		mapping	when	$1<p<2$.
%%%%%%%%%%%%%%%-----------
%,	see	Remark	\ref{Rmq-perrurbation-Lips}.	
%%%%%\tableofcontents              
%%%%Namely,	the	existence	of	an	ergodic	invariant	measure	is	given	in	Theorem	\ref{THm1}	if	$2<p<\infty$.	Then,		Theorem	\ref{Thm1-bis}		and	Theorem	\ref{Thm1-bis-2}	concern	the	existence	of	an	ergodic	invariant	measure	if	$\frac{2d}{d+2}<p\leq2	$.		Theorem	\ref{THm2}	concerns	the		uniqueness	of	invariant	measure.
%%%%%%%%%%%%%%---------
\section{Stochastic	obstacle	problems	$\&$	main	results	}\label{Sec-formulation}
\subsection{Notation and  function spaces}
Let us denote by $D \subset \R^d,d\geq1$ a Lipschitz bounded domain, $T>0$ and consider $\underline{\mathbf{p}}:=\max(1,\frac{2d}{d+2}) <p <+\infty$.   
As usual, $p^\prime=\frac{p}{p-1}$ denotes the conjugate exponent of $p$, $V=W^{1,p}_0(D)$, the sub-space of elements of $W^{1,p}(D)$ with null trace, endowed with Poincar\'e's norm,  $H=L^2(D)$ is identified with its dual space so that, the corresponding dual spaces to $V$, $V^\prime$ is $W^{-1,p^\prime}(D)$ and the Lions-Guelfand triple $\V \hookrightarrow_d H=L^2(D)   \hookrightarrow_d \Vprime$ holds. 
Denote by $p^*=\frac{pd}{d-p}$ if $p<d$ the Sobolev embedding exponent and remind that	(see	\textit{e.g}	\cite[Theorem	1.20]{Roubicek}) 
\begin{align*}
\text{ if }p<d,\quad &V \hookrightarrow L^a(D),\ \forall a \in [1,p^*]\text{ and compactly if }a \in [1,p^*),
\\
\text{ if }p=d,\quad & V \hookrightarrow L^a(D),\ \forall a <+\infty\text{ and compactly}, 
\\
\text{ if }p>d,\quad &V \hookrightarrow C(\overline D)\text{ and compactly }.
\end{align*}
Since $\underline{\mathbf{p}}=\max(1,\frac{2d}{d+2}) <p <+\infty$, the compactness of the embeddings hold in Lions-Guelfand triple.	The duality bracket for $T \in V^\prime$ and $v \in V$ is denoted by $\langle T,v\rangle$ and the scalar product in $H$ is denoted by $(\cdot,\cdot)$. 		We recall the existence of 	$C_D>0$		such that
		\begin{align}\label{const-embeeding}
		\Vert	u\Vert_H^2\leq	C_D\Vert	u\Vert_V^2,	\quad	\forall	u\in	V,	\text{	since	}	V\hookrightarrow	H.
		\end{align}

Let $(\Omega,\mathcal{F},P)$ be  a complete probability space  endowed with a  right-continuous filtration $\{\mathcal{F}_t\}_{t\geq 0}$\footnote{For	example,	$(\mathcal{F}_t)_{t\geq 0}$ is the augmentation of the filtration generated by  $\{W(s), 0 \leq s\leq t \}_{0\leq t\leq T}$.} completed with respect to the measure $P$.  $W(t)$ is	a	$\{\mathcal{F}_t\}_{t\geq 0}$-adapted 	$Q$-Wiener process in $H$,	where $Q$	is 	non-negative symmetric   operator  with	finite	trace	\textit{i.e.} $tr Q < \infty$. Denote by $\Omega_T=(0,T)\times \Omega$ and $\mathcal{P}_T$ the predictable $\sigma$-algebra on $\Omega_T$\footnote{$\mathcal{P}_{T}:=\sigma(\{ ]s,t]\times F_s \vert 0\leq s < t \leq T,F_s\in \mathcal{F}_s \} \cup \{\{0\}\times F_0 \vert F_0\in \mathcal{F}_0 \})$ (see \cite[p. 33]{RoK15}). Then, a process defined on $\Omega_T$ with values in a given space $E$ is predictable if it is $\mathcal{P}_{T}$-measurable.}. 
Set $H_0=Q^{1/2}H$ and we recall that $H_0$ is a separable Hilbert space endowed with the inner product $ (u,v)_0=(Q^{-1/2}u,Q^{-1/2}v)$, for any $u,v\in H_0$. The space	$(L_2(H_0,H), \Vert \cdot \Vert_{L_2(H_0,H)})$ stands	for the space of Hilbert-Schmidt operators from $H_0$ to $H$\footnote{	$\Vert T \Vert_{L_2(H_0,H)}^2=\displaystyle\sum_{k\in\mathbb{N}}\Vert	Te_k\Vert_H^2$	where	$	\{e_k\}_{k\in\mathbb{N}}$	is	an	orthonormal basis	for	$H_0$,	see	\textit{e.g.}	\cite[Appendix B]{RoK15}.}	and	$\E$	stands	for	the	expectation	with	respect	to	the	 probability measure	$P$. \\

We recall that an element $\xi\in V'$	(resp.	$L^{p'}(\Omega_T;V')$) is called \textit{non-negative} \textit{i.e.}	$\xi\in ( \Vprime)^+$(resp.	$(L^{p'}(\Omega_T;V'))^+$)	iff 
$	\langle \xi,\varphi\rangle_{V',V}\,\geq 0\quad	(\text{	resp.	}\E\int_0^T\langle \xi,\varphi\rangle_{V',V}\,dt\geq 0)$
holds for all $\varphi\in V$ (resp.	$L^{p}(\Omega_T;V)$)	such that $\varphi\geq 0$. In this case, with a slight abuse of notation, we will often write $\xi\geq 0$.	
Denote by 	$ V^* =( \Vprime)^+ - (\Vprime)^+ \varsubsetneq \Vprime$
the order dual as being the difference of two non-negative elements of $\Vprime$, more precisely $h\in V^*$ iff
$h=h^+-h^-$ with $h^+,h^- \in (\Vprime)^+$.

\subsection{Formulation of the problem}
	Our	aim	is	to	investigate	the	existence	and	uniqueness	of	invariant	measures			associated	with	 the following problem	
    \begin{align}\label{I-invariant----}
\begin{cases}
du+A(u)ds+\mathbf{F}(u)ds+ k ds= fds+ G(u)dW \quad & \text{in} \quad D\times  \Omega_T, \\ 
\langle k,u-\psi\rangle_{V^\prime,V}  = 0 \hspace*{0.2cm} \text{and} \hspace*{0.2cm} -k \in	(V^\prime)^+  & \text{a.e.	in} \quad  \Omega_T,
\\
u \geq \psi & \text{in} \quad  D\times  \Omega_T, \\  
u=  0   &\text{on} \quad\partial D\times \Omega_T,  \\
u(t=0)=u_0 & \text{in} \quad D,
\end{cases}
\end{align}

\subsubsection*{Assumptions}
We will consider in the sequel the following assumptions:
\begin{enumerate}
	\item[H$_1$] :
	Let $A: V \to V^\prime$, \ $G: H \to L_2(H_0,H)$	be		measurable 	such that:
		\item[H$_2$]:  there exist $ \alpha,\bar K >0,   \lambda,\lambda_T,l_1 \in \mathbb{R}$  such that:
	\begin{enumerate}
		\item[H$_{2,1}:$]  (Coercivity)   for all	$  v \in V,\quad \langle A(v),v\rangle+\lambda \Vert v\Vert_H^2+l_1 \geq \alpha \Vert v\Vert_V^p.$
		\item[H$_{2,2}:$]  ($T-$ monotonicity)    for all	$ v_1,v_2 \in V$, 
		\begin{align*}
		\lambda_T(v_1-v_2,(v_1-v_2)^+)_H+\langle A(v_1)-A(v_2),(v_1-v_2)^+\rangle \geq 0.
		\end{align*}
		Note that since $v_1-v_2 = (v_1-v_2)^+ - (v_2-v_1)^+$,  $\lambda_T Id+A$ is also monotone.
		\item[H$_{2,3}:$]   (Boundedness)  for all	$  v \in V,\quad  \Vert A(v)\Vert_{V^\prime} \leq \bar K(\Vert v\Vert_V^{p-1}+1).$

		\item[H$_{2,4}:$]   (Hemi-continuity)    for all $ v,v_1,v_2 \in V$: $ s \in \R \mapsto \langle A(v_1+s v_2),v\rangle$ is continuous.
		
	\end{enumerate}
	%
	%%%and $l\in L^1(\Omega_T)$, predictable,
	\item[H$_3$]:	 $\exists  L_G,	M >0$  such that:\\  for all $ \theta,\sigma \in H$:   $\Vert G(\theta)-G(\sigma)\Vert_{L_2(H_0,H)}^2 \leq L_G \Vert \theta-\sigma\Vert_H^2$	and	   $ \Vert G(0)\Vert_{L_2(H_0,H)}^2 \leq M.$

	\item[H$_4$] :	The	obstacle $\psi\in V$	such	that	$G(\psi)=0$.
	\item[H$_5$] : 	The	external	force $f\in V^\prime$ such  that 
	$
	 f  - A(\psi)=h   \in V^*.$
		
	\item[H$_6$] : The	initial	datum	 $u_0 \in H$,	and	 satisfies the constraint, \textit{i.e.} $u_0 \geq \psi$.
%%\footnote{ In the case  $\max(1,\frac{2d}{d+2})<p<	2$.}
\item[H$_7$] : Let 
     $\mathbf{F}: \mathbb{R} \to \mathbb{R}$, 	 there exist $L_F,K >0$ such that
	\begin{align}\label{H-F}
	 \text{  for all	}  x,y \in \mathbb{R}:   \vert \mathbf{F}(x)-\mathbf{F}(y)\vert \leq L_F \vert x-y\vert.	
	\end{align}
    %%\text{	and	}	    \vert \mathbf{F}(0)\vert:= K.
\end{enumerate}
Let us make some comments on H$_7$, which plays a role only if $\underline{\mathbf{p}}<p<	2$.
\begin{remark}\label{Rmq-perrurbation-Lips}
	Let	$v,v_1,v_2	\in	H$	and	note	that
	\begin{align*}
	\vert		( \mathbf{F}(v_1)-\mathbf{F}(v_2),(v_1-v_2)^+)_H\vert&=\vert\int_D[\mathbf{F}(v_1)-\mathbf{F}(v_2)].(v_1-v_2)1_{\{v_1-v_2\geq	0\}}dx\vert\\	&\leq	L_F\int_D\vert	v_1-v_2\vert^21_{\{v_1-v_2\geq	0\}}dx=L_F\Vert	(v_1-v_2)^+\Vert_H^2,
	\end{align*}
	and
	$\vert	(\mathbf{F}(v),v)_H\vert	\leq	(L_F+\dfrac{1}{2})\Vert	v\Vert_H^2+\dfrac{1}{2}\Vert	F(0)\Vert_H^2\leq	(L_F+\dfrac{1}{2})\Vert	v\Vert_H^2+\dfrac{CK}{2}$.	Furthermore,
	$\Vert	\mathbf{F}(v)\Vert_H\leq	L_F\Vert	v\Vert_H+CK.
	$
	Denote	by	$\widehat{A}(u)=A(u)+\mathbf{F}(u)$,		then	$\widehat{A}$	satisfies	$H_1,H_{2,1}-H_{2,4}$	if	$p\geq2$.	In	the	case	$1<p<2,$		$H_{2,3}$	is	not	satisfied.
	Finally,	thanks	to	$H_4$	one	has
	$\mathbf{F}(\psi)\in	H$.	Thus,	by	using	H$_5$,	we	get
	\begin{align}\label{new-decompsition}
	f  - A(\psi)-\mathbf{F}(\psi)=[h^++(\mathbf{F}(\psi))^-]-[h^-+ (\mathbf{F}(\psi))^+]=\widetilde{h}^+-\widetilde{h}^- =\widetilde{h} \in V^*.
	\end{align}
\end{remark}
Notice	that				\textit{e.g.}	the	perturbation	of	$A=-\Delta_p$;	$1<p<2$,		by	a	zero	order	Lipschitz	mapping	is	not	covered		by	the	well-posedness	result	\cite[Theorem	1]{YV22},	since	$H_{2,3}$	is	not	satisfied. On the other hand, 
the incorporation of zero order Lipschitz mapping in the operator part is important to study the invariant measures if	$\underline{\mathbf{p}}<p<	2$.		Therefore,	we	need	to	check	the	well-posedness	of	\eqref{I-invariant----},	which	is	given	in	Theorem	\ref{Thm-perturbation}.	Its	proof	follows	by	a	cosmetic	changes	of	the	one		\cite[Theorem	1]{YV22}	but for the convenience of the reader, we present the	proof	in		Section	\ref{sec-appendix}.	We	invite	the	interested	reader	by	more	general	situations	about	obstacle	problems	in	the	presence	of	a	multiplicative	noise	to	consult	the	recent	work	\cite{NYGA},	where	well-posedness	and	Lewy-Stampacchia's	inequality	were	proved	by	using	stochastic	compactness	tools.
\begin{remark}
Considering the different notions of operators, it is useful to make some comments on the relationship between T-monotone, maximal monotone, and accretive operators.
Let $ v_1,v_2 \in V,$ recall that $v_1-v_2=(v_1-v_2)^+-(v_2-v_1)^+$. Hence, it is not difficult to verify that H$_{2,2}$ implies that $A$ is  monotone, \textit{i.e.} 
  for all	$ v_1,v_2 \in V$, 
		\begin{align*}
		\lambda_T\Vert v_1-v_2\Vert_H^2+\langle A(v_1)-A(v_2),v_1-v_2\rangle \geq 0.
		\end{align*}
 Since $A$ is also hemicontinuous, then $A$ is maximal monotone in $V\times V^\prime$, see  \cite[Theorem 2.4]{BARBU}.  It is worth mentioning that the notion of T-monotone operator was introduced in \cite{Brezis-Stampacchia} to study elliptic inequalities and the regularity of its solutions, see also \cite{Mosco}. This is somewhat related to the fact that when solving variational inequalities or obstacle problems, the solution lies in a convex subset and not the full Banach space and one needs to handle non ordered Banach space. Moreover, note also that the monotonicity property is defined from $V$ to its dual space and it is a variational property, whereas  the the notion of accretivity is  a metric geometric property defined 
 for an operator $A$ from   a Banach space $X$ to  itself, see \textit{e.g.} \cite{BARBU} for more details.
\end{remark}

\begin{remark}
	Note	that	$G(\psi)=0$	in	
	$H_4$	ensures		that	Lewy-Stampacchia	inequalities	\eqref{LS}	holds,	which	allows	us	to	manage	the	reflection	measure	in	appropriate	sense	and	obtain		well- posedness	result	for	obstacle	problems with larger class of data,	we	refer	to	\cite{YV22}.	Moreover,	a	large	class	of	$V$-valued	obstacle 	problems	  can reduce to
the question of a positivity obstacle problem with a stochastic reaction term vanishing	at $0$,			see	\cite[Rmk.	3]{YV22}.		Finally,	it is  appropriate	in many situations as	well, from a physical point of view,	to	consider	$G$	vanishes	at	the	obstacle,	see	\textit{e.g.}	\cite{Bauz}.
\end{remark}

	%%%
	%%%
\subsection{Well-posedness	of		\eqref{I-invariant----}}
Consider	the	following	problem:
	find	$(u,k)$,	in the 	sense of Definition \ref{def-1},	solution	of	\eqref{I-invariant----}.	Denote by $K$ the convex  set of admissible solutions
	$$K=\{v \in L^p(\Omega_T;V),  \quad v(x,t,\omega)\geq \psi(x) \quad \text{a.e. in } D\times\Omega_T\}.$$
	%%Let us recall the concept of a solution for  	\eqref{I-invariant----}.
	\begin{definition}\label{def-1}
		The pair $(u,k)$	 is a solution to Problem \eqref{I-invariant----} if:
		\begin{enumerate}
			\item  $u:\Omega_T\to H$ and $k:\Omega_T \to V^\prime$ are predictable	processes. 
			\item  	 $u \in L^2(\Omega;C([0,T];H))$,	$u\in L^p(\Omega_T;V)$, $u(0,\cdot)=u_0$ and $u\geq \psi$, \textit{i.e.}, $u\in K$.
			\item P-a.s, for all $t\in[0,T]$: 
						\begin{align*}
			u(t)+\displaystyle\int_0^t [A(u)+\mathbf{F}(u)+k]ds=u_0+\displaystyle\int_0^t G(u)dW(s)+\displaystyle\int_0^t fds	\text{	in	}	V^\prime.\end{align*}			
			\item $- k\in (L^{p^\prime}(\Omega_T,V^\prime))^+$ \ and\quad $ \forall v \in K$,  $\langle k,u-v\rangle  \geq 0$ \ a.e. in $\Omega_T$.
				\end{enumerate}
	\end{definition}
    %%%&u(t)+\displaystyle\int_0^t [A(u)+k]ds=u_0+\displaystyle\int_0^t G(u)dW(s)+\displaystyle\int_0^t fds	\text{	in	}	V^\prime,\\\text{resp.	}	&
	\begin{remark}\label{Rmq-def}	The		point	(4)	in	Definition	\ref{def-1}	can	be	replaced	by
		$k \in L^{p^\prime}(\Omega_T; V^\prime)$	 with 
		\begin{align}\label{230530_01}
		- k\in (L^{p^\prime}(\Omega_T;V^\prime))^+ \ &\text{and} \ \E{\int_0^T\langle k,u-\psi\rangle_{V',V}\,dt}=0
				\end{align}
                %%%%
            %    ,\notag\\
			%- \rho\in 	((L^{p^\prime}(\Omega_T;L^{p^\prime}(D))\big)^+ \ &\text{and} \ \E{\int_0^T\langle \rho,u-\psi\rangle_{V',V}\,dt}=0.
                %%
		Condition \eqref{230530_01} can be understood as a minimality condition on $k$	 in the sense that $k$	 vanishes on the set $\{u>\psi\}$. Moreover, \eqref{230530_01} implies that, for	all	$	v \in K$,  
		$\E{\int_0^T\langle k,u-v\rangle\,dt}  \geq 0.$	
        %%%%(resp.	$\E{\displaystyle\int_0^T\langle %%%\rho,u-v\rangle\,dt}  \geq 0$).
			\end{remark}
	Let us present the   following	result about the solution to \eqref{I-invariant----}. We distinguish between the cases $2\leq p<+\infty$ and  $\underline{\mathbf{p}}<p<	2$.
		By using \cite[Theorem	1]{YV22}, we have	the	following	result.
	\begin{theorem}\label{TH1}($2\leq p<	+\infty$) Under the assumptions H$_1$-H$_7$, there exists  a unique solution $(u,k)$ to  \eqref{I-invariant----} in the sense of Definition \ref{def-1}. Moreover,  the following  Lewy-Stampacchia  inequality holds
		\begin{align}\label{LS}
		0\leq  \partial_t \Big( u-\displaystyle\int_0^\cdot G(u)dW\Big)+ A(u)-f=-k \leq h^-=  \left(f   - A(\psi)\right) ^-.
		\end{align}
	\end{theorem}

The	next	theorem	concerns	the	well-posedness	of	\eqref{I-invariant----}	if	  $1<p<2$,	see	Section	\ref{sec-appendix}.

\begin{theorem}\label{Thm-perturbation}($\underline{\mathbf{p}}<p<	2$)
Assume $H_1-H_7$	 hold and	assume	moreover	that	$\widetilde{h}^-$	is a  non negative element of $ L^{p^\prime}(D)$,		see	\eqref{new-decompsition}.	Then, there exists a unique 
	solution $(\widetilde{u},\widetilde{k})$ to  \eqref{I-invariant----} in the sense of Definition \ref{def-1}.
	Moreover,	$- \widetilde{k}\in 	((L^{p^\prime}(\Omega_T;L^{p^\prime}(D))\big)^+$  
  and the	following		holds
	\begin{align}\label{cv-pen-sol-1}
	\E\displaystyle\sup_{t\in[0,T]}\Vert (\widetilde{u}_\epsilon-\widetilde{u})(t)\Vert_H^2\leq   C\epsilon^{\frac{1}{p-1}}	\quad	\text{	for	all	}	 \epsilon>0,
	\end{align}
	where	$\widetilde{u}_\epsilon$	is	the	unique	solution	of	\eqref{1}.
    	 \end{theorem}
         \begin{remark} 
             We wish to draw the reader’s attention to the fact that the extra regularity of   $\widetilde{h}^-$ if $\underline{\mathbf{p}}<p<	2$ , namely	 $ \widetilde{h}^-\in (L^{p^\prime}(D))^+$  can be removed and consider   $ \widetilde{h}^-\in (V^\prime)^+$ and prove the well-posedness of  \eqref{1}. First, one  proves Lewy-Stampacchia inequality with   $ \widetilde{h}^-\in (L^{p^\prime}(D))^+$ by repeating the same arguments of \cite[Subsection 3.2]{YV22}. Then,  use Lewy-Stampacchia inequality to extend the well-posedness to  the case  $ \widetilde{h}^-\in (V^\prime)^+$, see \cite[Subsection 3.3]{YV22}. Therefore, the results can  be extended to the case of $ \widetilde{h}^-\in (V^\prime)^+$ by similar arguments.
         \end{remark}
	\begin{remark}\label{Rmq-distinguish}
	Notice that 	the unique  solution to \eqref{I-invariant----} in the cases $2\leq p<+\infty$ and  $\underline{\mathbf{p}}<p<	2$ satisfies the same properties of Definition \ref{def-1}.  Therefore, we do not distinguish between them  in the rest of the paper unless it is necessary to understand the argument.\end{remark}
	%%%%%%%%	$(y,\rho)\in L^p(\Omega_T,V) \times L^{p^\prime}(\Omega_T,L^{p^\prime}(D))$, both predictable,  satisfying: 
	%%%\begin{itemize}
	%%	\item  $y\in  L^2(\Omega,\mathcal{C}([0,T],H)) \cap K$ and $\rho \leq 0$.
	%%%	\item For any $t\in[0,T]$: $ y(t)+\displaystyle\int_0^t(A(y)+\mathbf{F}(y)+\rho-f)ds=u_0+\displaystyle\int_0^t G(y)dW(s).$
	%%%%	\item $\rho\in	L^{p^\prime}(\Omega_T\times D);	\langle \rho, y-\psi\rangle=0$ a.e. in $\Omega_T$  and, for any $v\in K$,  $\langle \rho, y-v\rangle\geq 0$ a.e. in $\Omega_T$.
	%%%%%\end{itemize}

	\subsection{Main	results}In this subsection, we collect the results of this work.
	First,	let	us	recall	the	following:		
		$L^2(H,\mu)$	denotes	the	Hilbert	space	of	all	equivalance	classes	of	Borel	square	integrable	real	functions	on	$H$	with	respect	to	the	Gaussian	measure	$\mu$,	$B_b(H)$	denotes	 the set of bounded Borel functions from $H$ to $\mathbb{R}$	and		denote	by	$\mathcal{C}_b(H)$	the	set	of	bounded continuous function	from	$H$	to	$\mathbb{R}$.	We	recall	that	$$		(\varphi,\Psi)_{L^2(H,\mu)}=\int_H\varphi\Psi	d\mu,	\quad	\varphi,\Psi\in	L^2(H,\mu),$$	we	refer	\textit{e.g.}	 to \cite[Chapter	9]{Daprato-lecture}	for	more	details.	
	\begin{definition}
		Let	$(P_t)_t$	be	a	Markov	semigroup	on	$H$.
	\begin{itemize}
		\item	A	probability	measure	$\nu$	is	said	to	be	invariant	for	$P_t$	if:	$$\int_HP_t\varphi	d\nu=\int_H\varphi	d\nu,	\quad	\forall	\varphi	\in	B_b(H)	\text{	and	}	t\geq	0.	$$
		\item	Let	$\nu$	be	an	invariant	measure	for	$P_t$.
        \begin{itemize}
            \item[i.]   $\nu$	is	ergodic	if:		$$	\displaystyle\lim_{T\to	+\infty}\dfrac{1}{T}\int_0^TP_t\varphi	dt=\int_H\varphi(x)	\nu(dx),\quad	\forall	\varphi\in	L^2(H,\nu).$$
	\item[ii.]	$P_t$	is	strongly	mixing	if:	
		$$	\displaystyle\lim_{t\to	+\infty}	P_t\varphi(x)=\int_H\varphi(x)	\nu(dx)	 \text{ in }  L^2(H,\nu)	\text{ where }	x\in	H.$$
        \end{itemize}
       	
	\end{itemize}
	\end{definition}

%\yass{******************************************\\}

%Let	us	denote	by	$K_\psi=\{	h\in	H,\quad	h\geq	\psi\}$	and	note	that	$K_\psi$	is	a		closed convex set	in	$H$.	For 	$t\geq	0$,	we	define	the	following	operators
%\begin{align}\label{transition-prob-1}
%P_t\varphi:K_\psi	\to	\mathbb{R};\quad
%(P_t\varphi)(\eta):=\E[\varphi(u(t;\eta))],	\quad	\eta\in	K_\psi	\quad	\forall	\varphi	\in	\mathcal{C}_b(H).
%\end{align}

%\yass{******************************************\\}
Now, denote	by	$K_\psi=\{	h\in	H,\quad	h\geq	\psi\}$	and	note	that	$K_\psi$	is	a		closed convex set	in	$H$.	
        Denote by $\pi_\psi$ the projection from $H$ onto the closed convex set $K_\psi$ \textit{i.e.} $\pi_\psi:H\to K_\psi$. It is known that $\pi_\psi$ is $1$-Lipschitz continuous. We denote by $\mathcal{C}_b(K_\psi)$ the space of all $\varphi:K_\psi \to \mathbb{R}$
 being bounded and uniformly continuous with respect to $\Vert \cdot\Vert_H$. Thus, we can identify $\mathcal{C}_b(K_\psi)$ with a subspace of $\mathcal{C}_b(H)$ by using the embedding $\mathcal{C}_b(K_\psi)\to \mathcal{C}_b(H); \varphi \mapsto \varphi\circ \pi_\psi.$\\
        
For 	$t\geq	0$,	we	define	the	following	Markov	Feller	semigroup	$(P_t)_{t\geq		0}$	(see	Lemma		\ref{Markov-Lemma}).
\begin{align}\label{transition-prob}
P_t\varphi:K_\psi	\to	\mathbb{R};\quad
(P_t\varphi)(\eta):=\E[\varphi(u(t;\eta))],	\quad	\eta\in	K_\psi	\quad	\forall	\varphi	\in	\mathcal{C}_b(K_\psi).
\end{align}
where	$u(t;\eta),	t\geq	0$	be	the	unique	strong	solution	to		\eqref{I-invariant----}		starting	from	$\eta	\in	K_\psi$. 
		Let us state our main results,	we distinguish between the  cases:	$2 \leq p<\infty$	and	$\underline{\mathbf{p}}<p<2$. Before that, let us introduce the following assumption
        \begin{align}\label{bound-noise-p=2}
            \text{There exists 	}  \mathbf{K} > 0:  \forall \sigma \in H:   	    \Vert G(\sigma)\Vert_{L_2(H_0,H)}^2 \leq \mathbf{K}.
        \end{align}
       
          The following result concerns a moment estimates of the set of invariant measures. Denote by $\Lambda$ the set of invariant measures  for	the	semigroup	$(P_t)_t$	defined	by	\eqref{transition-prob}, see propositions \ref{ergo-1-prop} and \ref{prop-ergodic-0} for the proof.
    \begin{prop}\label{prop-ergodic-full}
    Assume that	$H_1$-$H_7$ hold.
If $2\leq p<+\infty$,	then	 there exists $\widetilde{\mathbf{K}}>0$ such that $
\int_H\Vert	x\Vert_V^p\mu(dx)	\leq \widetilde{\mathbf{K}}, \forall \mu \in \Lambda.$
          If $\underline{\mathbf{p}} < p<2$, then	 there exists  $\bar{\mathbf{K}}>0$ such that 
  $\int_H\Vert	x\Vert_H^p\mu(dx)	\leq \bar{\mathbf{K}}, \forall \mu \in \Lambda.	
$
\end{prop}
	Now, we present the results about the existence and uniqueness of invariant measures.
        \begin{theorem}\label{THm1}(case	$2\leq p<\infty$)	Assume that	$H_1$-$H_6$ hold and $\mathbf{F}\equiv 0$ in $H_7$\footnote{In this case, $\mathbf{F}\equiv 0$ is not a restriction, since one needs only to change $\lambda$ by $\lambda+L_F$. }.
    \begin{enumerate}
        \item Let $2<p<+\infty,$    there exists	an	 ergodic	invariant	measure	$\mu$	for	$(P_t)_t$	associated	with	\eqref{I-invariant----}.
		Moreover,	$\mu$	is	concentrated	in	$V\cap	K_\psi$	and		satisfying	
	$\int_H\Vert	x\Vert_V^p\mu(dx)	<	+\infty.
	$
    \item  Let	$K>0$ and 	$p=2$	and			assuming
		that	there	exists	$	\bar{\alpha}	>0$		such	that		\begin{align}\label{condition-invariant}
		(1-\delta)\alpha-C_D[\lambda+\dfrac{L_G(1+K^2)}{2K^2}]^+\geq	\bar{\alpha}	\text{	for	some	}\delta\in]0,1[,	\end{align}
		then,	there	exists	an	 ergodic	invariant	measure	$\mu$	for	$(P_t)_t$	associated	with	\eqref{I-invariant----}, 	concentrated	in	$V\cap	K_\psi$	and		satisfying	
$	\int_H\Vert	x\Vert_V^2\mu(dx)	<	+\infty.
$	 \item Let $p=2$, the same result holds if one replaces 
		\eqref{condition-invariant} by 	 $\lambda\leq	0 \text{ in }	H_{2,1}$ and \eqref{bound-noise-p=2}.
        
        %%%%%%\begin{align}\label{bound-noise-p=2}
	%\lambda\leq	0 \text{ in }	H_{2,1} \text{ and 	there %%%%exists 	}  \mathbf{K} > 0:  \forall \sigma \in H:   %%%	    \Vert G(\sigma)\Vert_{L_2(H_0,H)}^2 \leq %%%%\mathbf{K}.	\end{align}
    \end{enumerate}
 	\end{theorem}
  	\begin{theorem}\label{THm2}
	Under	the	assumptions	of	Theorem	\ref{THm1}.	
    \begin{enumerate}
        \item If   $\dfrac{L_G}{2}+\lambda_T<0$  in the case (1) and (2).  	Then	there	exists	a	unique	ergodic	and	strongly mixing	invariant	measure	$\mu$	and	the	following convergence to equilibrium holds.
	\begin{align}\label{cv-equilibruim}
	\Big	\vert	P_t\varphi(x)-\int_H\varphi(y)\mu(dy)\Big\vert	\leq	\mathcal{G}\Vert	\varphi\Vert_{1,\infty}	e^{[\frac{L_G}{2}+\lambda_T]t},\quad	\forall	\varphi\in	\mathcal{C}_b^1(H).
	\end{align}
\item    If       $\lambda_T<0$  in the case (3). Then
Then	there	exists	a	unique	ergodic	and	strongly mixing	invariant	measure	$\mu$	and	the	following convergence to equilibrium holds.
	\begin{align}\label{cv-equilibruim-2-case}
	\Big	\vert	P_t\varphi(x)-\int_H\varphi(y)\mu(dy)\Big\vert\leq	\Vert	\varphi\Vert_{1,\infty}	(\mathcal{G}+\sqrt{\mathbf{K}t}) e^{\lambda_T t} \quad	\forall	\varphi\in	\mathcal{C}_b^1(H),
	\end{align}
    \end{enumerate}
    where $\mathcal{G}$ depends only on $\widetilde{\mathbf{K}}$ and $ \bar{\mathbf{K}}.$
\end{theorem}
%\begin{remark}
%It is worth recalling that
%the	proof	of	uniqueness	of	invariant	measure	requires		"the	strict monotonicity",	namely
	%the	condition	$\frac{L_G}{2}+\lambda_T<0$.		In	the	case	of	equations,	we	refer	\textit{e.g.} to	\cite[Section	4.3]{Prevot-Rockner}.
%\end{remark}
Concerning the case $\underline{\mathbf{p}}<p<2$, we have
\begin{theorem}\label{Thm1-bis}
	Let	$\underline{\mathbf{p}}<p<2$, under the  assumptions	$H_1$-$H_7$,	let	$K>0$,	the	following	holds.
	\begin{enumerate}	
		\item  Under	the	assumptions		$\frac{L_G(1+K^2)}{2K^2}+\lambda+L_F\leq 0.	$		There	exists	an invariant	measure	$\mu$	for	$(P_t)_t$	associated	with	\eqref{I-invariant----}.	

				\item	Assume that	$\lambda+L_F\leq	0	$
			and $G$ satisfies \eqref{bound-noise-p=2}.
			Then	there	exists	an		invariant	measure	$\mu$	for	$(P_t)_t$	associated	with	\eqref{I-invariant----}.
	\end{enumerate}
	Moreover,	$\mu$	is	concentrated	in	$V\cap	K_\psi$	and		satisfying	
$\int_H\Vert	x\Vert_V^p\mu(dx)	<	+\infty.
$	

\end{theorem}
\begin{theorem}\label{THm2-2}
	Under	the	assumptions	 Theorem	\ref{Thm1-bis}.	  \begin{enumerate}
        \item If   $\dfrac{L_G}{2}+\lambda_T+L_F<0$  in the case (1).  	Then	there	exists	a	unique	ergodic	and	strongly mixing	invariant	measure	$\mu$	and	the	following convergence to equilibrium holds.
	\begin{align}\label{cv-equilibruim-00}
	\Big	\vert	P_t\varphi(x)-\int_H\varphi(y)\mu(dy)\Big\vert	\leq	\mathcal{G}\Vert	\varphi\Vert_{1,\infty}	e^{[\frac{L_G}{2}+\lambda_T+L_F]t},\quad	\forall	\varphi\in	\mathcal{C}_b^1(H).
	\end{align}
\item    If       $\lambda_T+L_F<0$  in the case (2). Then
Then	there	exists	a	unique	ergodic	and	strongly mixing	invariant	measure	$\mu$	and	the	following convergence to equilibrium holds.
	\begin{align}\label{cv-equilibruim-2-case-0}
	\Big	\vert	P_t\varphi(x)-\int_H\varphi(y)\mu(dy)\Big\vert\leq	\Vert	\varphi\Vert_{1,\infty}	(\mathcal{G}+\sqrt{\mathbf{K}t}) e^{\lambda_T t} \quad	\forall	\varphi\in	\mathcal{C}_b^1(H),
	\end{align}
    \end{enumerate}
    where $\mathcal{G}$ depends only on $\widetilde{\mathbf{K}}$ and $ \bar{\mathbf{K}}.$
	
\end{theorem}
\begin{remark}\label{Rmq-2cases}
Since we are	interested	in	the	existence	of	invariant	measures	to	a	Markov	semigroup	associated	with		non	linear	SPDEs		in	the	presence	of	a	multiplicative	noise,	where	the	solution	takes	values	in	the	Sobolev	space	$V$,		one	 expects	a restriction on	the	Lipschitz	constant	$L_G$,	see	$H_{3}$	to	obtain	the	existence	of	invariant	measures.	We		distinguish	two	cases.
	  If	the	dissipation	generated	by	the	T-monotone	operator	is strong enough,	then	the	restriction	on	the	Lipschitz	constant	$L_G$	is	not	needed,	see	Theorem	\ref{THm1}$_{(1)}$.	Otherwise,		a	restriction	on		$L_G$		is	needed,	see	\eqref{condition-invariant} and Theorem	\ref{Thm1-bis}. 	
	
	%\begin{enumerate}
	%\item		
  
    %\item		
		%consider	the		$p$-Laplace	operator	$\Delta_p$,		defined	as		$\Delta_pu=-\text{div}[\vert	\nabla	u\vert^{p-2}\nabla	u],\quad	u\in		V$	and	$	1<p<\infty$.
		%\begin{itemize}
			%\item			$2<p<\infty$:		the	dissipation	generated	by	$\Delta_p$	is		strong	and		we	can	prove	the	existence	of	an	ergodic	invariant	measure	without	any	restriction	on	$L_G$.
			%\item	$1<p\leq	2$:			$\Delta_p$		is	more	singular	and	
			%a	restriction	on	the	Lipschitz	constant	$L_G$	or	the	Hilbert-Schmidt	norm	of	the		operator	$G$	is	needed	to	obtain	the	existence	of	invariant	measures,		see	Theorem	\ref{Thm1-bis}	and	Theorem	\ref{Thm1-bis-2}...\yass{add the fact of multiplicative noise}\\
				%	\end{itemize}
		
	%						\end{enumerate}	
\end{remark}
		
%\subsection{Examples}\label{Examples}
	Consider	the	following	example.	
\begin{ex}
	Let	$\underline{\mathbf{p}}<p<\infty$,	$T>0$,	and		$(u,-k) \in (L^2(\Omega;C([0,T];H))\cap	L^p(\Omega_T;V))\times(L^{p^\prime}(\Omega_T,V^\prime))^+$	such	that	$u\geq	0$,
	be	the	unique	solution		of	the	following	problem
	\begin{align}\label{example1}
	\begin{cases}
	du-\text{div}[\vert	\nabla	u\vert^{p-2}\nabla	u]ds+\kappa	uds+ k ds= fds+ \mathbf{c}	ud\beta_s \quad & \text{in} \quad D\times  \Omega_T, \\ 
	u(\cdot,0)=u_0 \in	H,\quad		u_0 \geq 0;	\quad	f(x)=\sin(x)& \text{in} \quad D,\\ 
	\langle k,u\rangle_{V^\prime,V}  = 0 \hspace*{0.2cm} \text{and} \hspace*{0.2cm} -k \in	(V^\prime)^+  & \text{a.e.	in} \quad  \Omega_T,
	\end{cases}
	\end{align}
	where	$\kappa,\mathbf{c}	\in	\mathbb{R}$,	$\beta_t$	is a	one	dimensional brownian motion.
	Indeed, we can verify that	$A(u)=-\text{div}[\vert	\nabla	u\vert^{p-2}\nabla	u]$	satisfies	$H_1,H_{2,1},	H_{2,3}$	and	$H_{2,4}$	by	repeating those used in		\cite[Example	4.1.9]{Prevot-Rockner}.	We need only to verify	$H_{2,2}$,	the	T-monotonicity.	Denote	by	$1_{\{u>v\}}$	the	 indicator function	of	the	set	$\{x\in	D;		u(x)\geq	v(x)	\}$.
	Let	$u,v\in	V$	and	note	that,	for	any	$p>1$,	we	have
	\begin{align*}
	&	\langle	A(u)-A(v),(u-v)^+\rangle=-\langle	\text{div}[\vert	\nabla	u\vert^{p-2}\nabla	u]-\text{div}[\vert	\nabla	v\vert^{p-2}\nabla	v],(u-v)^+\rangle\\
		&=\int_D[\vert	\nabla	u(x)\vert^{p-2}\nabla	u(x)-\vert	\nabla	v(x)\vert^{p-2}\nabla	v(x)]\cdot	\nabla(u(x)-v(x))1_{\{u>v\}}(x)dx\\
		&\geq		\int_D(\vert	\nabla	u(x)\vert^{p-1}-\vert	\nabla	v(x)\vert^{p-1})\cdot	(\vert	\nabla	u(x)\vert-\vert	\nabla	v(x)\vert)1_{\{u>v\}}(x)dx\\
		&\geq	0,	\text{	since	}	\mathbb{R}_+\ni	s\mapsto	s^{p-1}	\text{	is	an	increasing	function.	}
	\end{align*}
	Hence	$H_{2,2}$	is	satisfied.
	An	application	of	Theorems	\ref{THm1},	\ref{THm2}, \ref{Thm1-bis}	 and \ref{THm2-2} yields
	\begin{itemize}
		\item	The	solution	$u$	of	\eqref{example1}		defines a	homogeneous
		Feller–Markov process.	
		\item	If	$p>2$:	there	exists	an	ergodic	invariant	measure	$\mu$	for	$(P_t)_t$	associated	with	\eqref{example1},		for	any	$\kappa,\mathbf{c}	\in	\mathbb{R}$.
		If	moreover	$\dfrac{	\mathbf{c}^2}{2}	<	\kappa$,	there	exists	a	unique	ergodic	and	strongly	mixing	invariant	measure	$\mu$	for	$(P_t)_t$	associated	with	\eqref{example1}.
		%	\item	If	$p=2$	and	$\dfrac{(1+n^2)}{2n^2}\mathbf{c}^2	\leq	%%\kappa,$	for	some	$n\in	\mathbb{N}^*$:	there	exists	a	unique	%%ergodic	and	strongly	mixing	invariant	measure	$\mu$	for	$(P_t)_t$	%%associated	with	\eqref{example1}.***Fix	the	value	of	n
		\item	If	$p=2$	and	$2\mathbf{c}^2	\leq	\kappa$:	there	exists	a	unique	ergodic	and	strongly	mixing	invariant	measure	$\mu$	for	$(P_t)_t$	associated	with	\eqref{example1}.
		
		\item	If		$	\underline{\mathbf{p}}<p<2$	and	
		$2\mathbf{c}^2<	\kappa:$		there	exists	a	unique	ergodic	and	strongly	mixing	invariant	measure	$\mu$	for	$(P_t)_t$	associated	with	\eqref{example1}.
	\end{itemize}	
\end{ex}
%\begin{remark}
%Note	that	$	A(u)=-\text{div}[\vert	\nabla	u\vert^{p-2}\nabla	u],	2<p<\infty$	  degenerates if	$\nabla	u\equiv0$.	
%%%%and in particular in the contact set	$\{u=	0\}$.		(see	\cite[Theorem	1]{YV22})
%\end{remark}

\begin{remark}\label{rmq-singular-plaplace}
	It's	worth	to	draw	the	reader	attention	to	the	following:	in	1D		and	2D	cases,	Theorem	\ref{Thm1-bis}	ensures	the	existence	of	an	invariant	measure	for	a	stochastic	obstacle	problems	governed	by	the	singular	$p$-Laplace	operator,	when	$1<p<2.$
	%%%%%%	and	$\frac{6}{5}<p<2$	in	3D	case.
\end{remark}
\section{Proofs}\label{Sec-Proof}
%Let us introduce the concept of a solution for Problem \eqref{I}.
\subsection{Construction	of	solution	to	\eqref{I-invariant----}}
Let	us	recall	here	the	main	steps	to	construct	the	solution	to	\eqref{I-invariant----}. We	refer	to	\cite{YV22}	for	the	detailed	proofs, see also Section  \ref{sec-appendix}.
\begin{enumerate}
	\item		Let $\epsilon>0$ and 		$u_0\geq	\psi$,
	consider the following  penalized problem:
	\begin{align}\label{1-penalization}
%%	\left\{ \begin{array}{l}
	u_\epsilon(t)+\displaystyle\int_0^t(A(u_\epsilon)-\dfrac{1}{\epsilon}[(u_\epsilon-\psi)^-]^{\tilde{q}-1}-f)ds=u_0+\displaystyle\int_0^t\widetilde{G}(u_\epsilon)dW(s) 
%%%	\end{array}
%%%	\right.
	\end{align}
	where $\tilde{q}=\min(p,2)$ and $\widetilde{G}(u_\epsilon,\cdot)=G(\max(u_\epsilon,\psi))$. By \cite[Thm. 4.2.4]{RoK15},	we	get
	\begin{lemma}\label{lemma-exis-u-epsilo,}
 For all $\epsilon>0$, there exists  a unique solution $u_\epsilon \in L^p(\Omega_T;V)$ predictable such that $u_\epsilon \in L^2(\Omega;C([0,T],H))$  and satisfying \eqref{1-penalization} P-a.s. in $\Omega$ for all $t\in[0,T]$. 
	\end{lemma}
\item	(A	priori	estimates)
	According	to \cite[Lemma 2]{YV22},  we  have
			 $(u_\epsilon)_{\epsilon>0}$ 	and	$(A(u_\epsilon))_{\epsilon>0}$	are bounded in  $ L^p(\Omega_T,V)\cap C([0,T],L^2(\Omega,H))$	and	$L^{p^\prime}(\Omega_T,V^\prime)$	respectively. 
	
	\item	(Convergence	with	regular	data) From \cite[Lemma 3	$\&$	Lemma	4]{YV22},	under	the	assumption  $h^-$  non negative element of $ L^{\tilde{q}^\prime}(D)$,	we	have
		\begin{itemize}
		\item
			$(\dfrac{1}{\epsilon}[(u_\epsilon-\psi)^-]^{\tilde{q}-1})_{\epsilon>0}$ is bounded in $L^{\tilde{q}^\prime}(\Omega_T\times D).$
			\item	$(u_\epsilon)_{\epsilon>0}$ is  a Cauchy sequence in the space $L^2(\Omega;C([0,T],H))$.	Moreover
			\begin{align}\label{cv1}
			\E\displaystyle\sup_{t\in[0,T]}\Vert (u_\epsilon-u_\delta)(t)\Vert_H^2 \leq	C(T)\epsilon	\quad	\text{	for	} 1>\epsilon\geq \delta >0.
						\end{align}
						
			\item	From	\cite[Lemma	5	$\&$	Theorem	2]{YV22},	there exists  a unique predictable stochastic process $(u,k)\in L^p(\Omega_T;V) \times L^{\tilde{q}^\prime}(\Omega_T;L^{\tilde{q}^\prime}( D))$	and		$(u_\epsilon)_{\epsilon>0}$	converges	strongly	in	$L^2(\Omega,C([0,T],H))$	to	$u$,	where	$u$	satisfying	
			 			\begin{itemize}
				\item[i.] $u \in  L^2(\Omega;C([0,T],H)) \cap K$, $ u(0)=u_0$.
				\item[ii.] $k\leq 0$ and $\forall v \in K$,  $ \langle k,u-v\rangle  \geq 0$ a.e. in $\Omega_T$. 
				\item[iii.] P-a.s, for all $t\in[0,T]$:  %\textcolor{red}{Attention \`a l'ordre des quatificateurs, pour tout $t$, p.s. indique que le compl\'ementaire n\'egligeable dans $\Omega$ d\'epend de $t$, ce qui peu \^etre \'evit\'e par continuit\' e en temps. }
				$$u(t)+\displaystyle\int_0^t [A(u)+k]ds=u_0+\displaystyle\int_0^t G(u)dW(s)+\displaystyle\int_0^t fds.$$ 
				\item[iv.]The   Lewy-Stampacchia inequality  $0\leq  -k \leq h^-=  \left(f   - A(\psi)\right) ^-$	holds.

			%%%%%	 $$0\leq  \partial_t ( u-\displaystyle\int_0^\cdot G(u)dW)+ A(u)-f \leq h^-=  \left(f   - A(\psi)\right) ^-.$$
			\end{itemize}
		\end{itemize}
			\item	(Convergence	with	general	data)	Following	\cite[Subsection	3.3]{YV22},		under	the	assumption	$h^-	\in	(V^\prime)^+$,	there exists $h_n \in L^{\tilde{q}^\prime}(\Omega_T; L^{\tilde{q}^\prime}(D))$ predictable  and non negative such that
			$ h_n \longrightarrow h^- \quad \text{in} \quad L^{p^\prime}(\Omega_T;V^\prime).$	Denote by $(u_n,k_n)$ the sequence of solutions given by (2) where $h^-$ is replaced by $h_n$	and	recall	that
			 \begin{itemize}
			 	\item	 Associated with $h_n$, denote the following $f_n$ by,
			 	$
			 	f_n=A(\psi)+h^+-h_n,  h^+\in (V^\prime)^+ . 
			 	$
			 	Note that $f_n \in V^\prime$ and $f_n$ converges strongly to $f$ in $ V^\prime$.
			 	
			 	\item	By Lewy-Stampacchia inequality, one has $0\leq -k_n\leq h_n.$
			 	\item		 $(u_n)_{n}$ is  a Cauchy sequence in the space $L^2(\Omega;C([0,T],H))$	and	
			 	\begin{align}\label{cv-2}
			 		 \E\displaystyle\sup_{t\in[0,T]}\Vert (u_n-u_m)(t)\Vert_H^2\leq	C\Vert f_n-f_m\Vert_{L^{p^\prime}(\Omega_T;V^\prime)} \leq  C\epsilon	\text{	for	large	}	 n \text{	and	} m.			 	\end{align}
			 	\item	$(u_n)_{n>0}$	converges	strongly	in	$L^2(\Omega;C([0,T],H))$	to	$u$,	solution	of	to	\eqref{I-invariant----} with $\mathbf{F}\equiv 0$	in	the	sense	of	Definition	\ref{def-1}.	
			 				  				 \end{itemize}

\end{enumerate}
From	\eqref{cv1}	and	\eqref{cv-2},	we	are	able	to	infer
\begin{align}\label{cv-pen-sol}
\E\displaystyle\sup_{t\in[0,T]}\Vert (u_\epsilon-u)(t)\Vert_H^2\leq   C\epsilon	\quad	\text{	for	all	}	 \epsilon>0,
\end{align}	
where	$u_\epsilon$	and	$u$,		respectively	are	the		unique	solution	of	\eqref{1-penalization}		and		\eqref{I-invariant----} with $\mathbf{F}\equiv 0$,	respectively.	From	Theorem	\ref{TH1perturbation},	we	have
	\begin{align}\label{cv-pen-sol--}
\E\displaystyle\sup_{t\in[0,T]}\Vert (\widetilde{u}_\epsilon-\widetilde{u})(t)\Vert_H^2\leq   C\epsilon^{\frac{1}{p-1}}	\quad	\text{	for	all	}	 \epsilon>0,
\end{align}
where	$\widetilde{u}_\epsilon$	and	$\widetilde{u}$,		respectively	are	the		unique	solution	of	\eqref{1}		and		\eqref{I-invariant----},	respectively.

\subsection{Feller-Markov process	associated		with	\eqref{I-invariant----}}
Let	$\epsilon>0$	and		$u_\epsilon(t,s;\xi),	t\geq	s\geq	0$	be	the	unique	strong	solution	to	\eqref{1-penalization}	starting at time $s$ from  $\mathcal{F}_{s}$-measurable initial data	$\xi	\in	H$.	Let $T > 0$,	we	have	
\begin{align}\label{Feller-penalization}
	 \E\Vert u_\epsilon(t,s;\xi)-u_\epsilon(t,s;\theta)\Vert_H^2 \leq  e^{[	L_G+2\lambda_T]	(t-s)}\E\Vert \xi-\theta\Vert_{H}^2	\quad	\forall	0\leq	s\leq	t\leq	T.
\end{align}
	Indeed,	by	using	\eqref{1-penalization}	and	applying Ito's formula	we	get
		\begin{align*}
	&\dfrac{1}{2}\Vert u_\epsilon(t,s;\xi)-u_\epsilon(t,s;\theta)\Vert_H^2-\dfrac{1}{2}\Vert \xi-\theta\Vert_{H}^2\\
	&+\displaystyle\int_s^t \langle A(u_\epsilon(r,s;\xi))-A(u_\epsilon(r,s;\theta)),u_\epsilon(r,s;\xi)-u_\epsilon(r,s;\theta)\rangle dr\\
	&+\displaystyle\int_s^t\langle  -\dfrac{1}{\epsilon}[(u_\epsilon(r,s;\xi)-\psi)^-]^{\tilde{q}-1}+\dfrac{1}{\epsilon}[(u_\epsilon(r,s;\theta)-\psi)^-]^{\tilde{q}-1}, u_\epsilon(r,s;\xi)-u_\epsilon(r,s;\theta) \rangle dr\\
	&= \displaystyle\int_s^t\langle (\widetilde{G}(u_\epsilon(r,s;\xi))-\widetilde{G}(u_\epsilon(r,s;\theta))), u_\epsilon(r,s;\xi)-u_\epsilon(r,s;\theta)\rangle	dW(r)\\
	&\quad +\dfrac{1}{2}\displaystyle\int_s^t\Vert \widetilde{G}(u_\epsilon(r,s;\xi))-\widetilde{G}(u_\epsilon(r,s;\theta))\Vert_{L_2(H_0,H)}^2dr.
	\end{align*}
	By	taking	the	expectation,	using	that	$x\mapsto	-(x)^{\tilde{q}-1}$		is non-decreasing,	$H_{2,2}$	and		$H_{3,1}$		we	get
		\begin{align*}
	&\E\Vert u_\epsilon(t,s;\xi)-u_\epsilon(t,s;\theta)\Vert_H^2 \leq\E\Vert \xi-\theta\Vert_{H}^2+(L_G+2\lambda_T)\E\displaystyle\int_s^t\Vert u_\epsilon(r,s;\xi)-u_\epsilon(r,s;\theta)\Vert_H^2dr,
	\end{align*}
	and	 \eqref{Feller-penalization}	follows from Gr\"onwall's lemma.
	\begin{remark}($\underline{\mathbf{p}}<p<2$)\label{Rmq-semigroup-second-case}
	It's	clear	that	\eqref{Feller-penalization}	holds	for	the	solution	to	\eqref{1},	denoted	by	$\widetilde{u}_\epsilon$,	with		$\lambda_T$ replaced by $\lambda_T+L_F$.
	\end{remark}
		Let	$\varphi\in\mathcal{C}_b(H)$,
we  define the operators	%$P_t^\epsilon:\mathcal{C}_b(H)\to	\mathcal{C}_b(H)$	by\footnote{$u_\epsilon:=\widetilde{u}_\epsilon$	in	this	case.}
\begin{align*}
P_t^\epsilon\varphi:H\to	\mathbb{R};\quad
(P_t^\epsilon\varphi)(\xi):=\E[\varphi(u_\epsilon(t;\xi))],	\quad	\xi\in	H,
\end{align*}
where		$u_\epsilon(t;\xi),	t\geq	0$	is	the	unique	strong	solution	to	\eqref{1-penalization}	(	or	\eqref{1} if $\underline{\mathbf{p}}<p<2$)	starting	from	$\xi\in	H$.	Thanks	to	Lemma	\ref{lemma-exis-u-epsilo,}(	see	Lemma	\ref{lemma-exis-u-epsilo,2}	for	the	case	of	\eqref{1})	and		by	using	a	similair	arguments	to	the	one	used	in	\cite[Section	9.6]{Peszat-Zabczyk}(see	also	\cite[Section	4.3]{RoK15}),		we	obtain	
	the	following	properties	of	$(P^\epsilon_t)_t$:
\begin{align}
&P_{s,t}^\epsilon=P_{0,t-s}^\epsilon	\quad	\forall 0\leq s \leq	t < \infty	\text{	and	}	P_{t}^\epsilon:=P_{0,t}^\epsilon,\label{stationary-eps}\\[0.2cm]
		&\E[\varphi(u_\epsilon(t+s;\eta))\vert	\mathcal{F}_t]=(P_s^\epsilon\varphi)(u_\epsilon(t;\eta))	\quad	\forall	\varphi\in	\mathcal{C}_b(H),	\forall	\eta\in	H,	\forall	t,s>0,\label{Markov-epsilon}\\[0.2cm]
	&(P_{t+s}^\epsilon\varphi)(\eta)=(P_{t}^\epsilon	(P_{s}^\epsilon\varphi))(\eta)	\quad	\forall	\varphi\in	\mathcal{C}_b(H),	\forall	\eta\in	H,	\forall	t,s>0,\label{semigroup-epsion}
\end{align}
where	$	P_{s,t}^\epsilon\varphi(\xi):=\E[\varphi(u_\epsilon(t,s;\xi))],\quad	\xi	\in	H$.
Therefore,	$u_\epsilon(\cdot;\xi),	\xi\in	H$	is	a	family of	Markov	processes
on	the	state	space	$H$	with	respect	to	the	filtration	$(\mathcal{F}_t)_t$,
	and	the	transition semi-group	$(P_t^\epsilon)_{t\geq	0}$	is	Feller		semigroup	thanks	to	\eqref{Feller-penalization}.	Thus,	\eqref{1-penalization}		defines a	homogeneous
	Feller–Markov process.	
\\

Now,	Let	us	show	that	\eqref{I-invariant----}		defines a	homogeneous
Feller–Markov process.		Let	$\eta	\in	K_\psi$		and	$u(t;\eta),	t\geq	0$	be	the	unique	strong	solution	to	\eqref{I-invariant----}	and	recall	that	$	u(t;\eta)	\in	K_\psi$ P-a.s.\\

%%%%the case non null F
%%%Similarly denote by	 $y(t,s;\xi);		t\geq	s\geq	0$ the	unique	strong	solution	to	\eqref{I-invariant----} 	starting at time $s$ from an $\mathcal{F}_{s}$-measurable initial data	$\xi	\in	H;	\xi\geq		\psi$.
In the following we focus on the case $\mathbf{F}\equiv 0$ and the case $\mathbf{F}\neq 0$ follows in the same way.
Denote	by	$u(t,s;\xi);t\geq	s\geq	0	$	the	unique	strong	solution	to	\eqref{I-invariant----} with $\mathbf{F}\equiv 0$		starting at time $s$ from an $\mathcal{F}_{s}$-measurable initial data	$\xi	\in	K_\psi$.	Thanks	to	\eqref{cv-pen-sol}	and	\eqref{cv-pen-sol--},	we	can	pass	to	the	limit	as	$\epsilon\to0$	in	\eqref{Feller-penalization}	and	deduce	
\begin{align}
\E\Vert u(t,s;\xi)-u(t,s;\theta)\Vert_H^2 \leq    e^{[	L_G+2\lambda_T]	(t-s)}\E\Vert \xi-\theta\Vert_{H}^2	\quad	\forall	0\leq	s\leq	t\leq	T,\label{Feller-obstacle}%\\
%\E\Vert y(t,s;\xi)-y(t,s;\theta)\Vert_H^2 \leq    e^{[	L_G+2(L_F+\lambda_T)]	(t-s)}\E\Vert \xi-\theta\Vert_{H}^2	\quad	\forall	0\leq	s\leq	%%t\leq	T.\label{Feller-obstacle-}
\end{align}
Next,	we	present	the	following	result	about	$(P_t)_t$	defined	in	\eqref{transition-prob}.
\begin{lemma}\label{Markov-Lemma}	$(P_t)_t$	is	bounded	on	$\mathcal{C}_b(K_\psi)$, Markov-Feller 	and	stochastically	continuous		semigroup	on		$\mathcal{C}_b(K_\psi)$.	Secondly,		for	$\varphi\in	\mathcal{C}_b(K_\psi),	\eta\in	K_\psi,$	and	$	t,s\geq	0$,	it	holds	that				\begin{align}
		&\displaystyle\lim_{\epsilon	\to	0}(P_t^\epsilon\varphi)(\eta)=(P_t\varphi)(\eta)\label{semigroup-cv},\\
		&(P_{t+s}\varphi)(\eta)=(P_{t}	(P_{s}\varphi))(\eta).	\label{Markov-obstacle}
	\end{align}
	%%%%%%%%%%%REMARK%%%%%%%%%%%%%%%%%%%%%%%%%%%%
	%%%%%%%%%%%\quad	\text{	and	}\quad
%%%%%%	\E[\varphi(u(t+s;\eta))\vert	\mathcal{F}_t]=(P_s\varphi)(u(t;\eta)).
%%%%%%%%%%%%%%%%%%%%%%%%%%%%%%%%%%%%ù
%%%%%%%%%%%%%%%%%%%%%%%%%%%%%%%%%%%
Moreover,	$P_{t+s,t}=P_{s,0}$		where	$P_{t,s}\varphi(\eta):=\E[\varphi(u(t,s;\eta))]$	for $0\leq s \leq	t < \infty,	$	and	$P_{t}:=P_{0,t}$.	
	\end{lemma}
\begin{proof}Let	$\varphi\in	\mathcal{C}_b(K_\psi),	\eta	\in	K_\psi$,	then
	%\begin{itemize}
				$\vert	P_t\varphi(\eta)\vert\leq	\Vert	\varphi\Vert_{\infty}<+\infty	$	and	therfore	$(P_t)_t$	is	bounded.	Secondly,
		we	recall	that	$(P_t^\epsilon\varphi)(\eta)=\E[\varphi(u_\epsilon(t;\eta))]$	and	$(P_t\varphi)(\eta)=\E[\varphi(u(t;\eta))]$.	Then,	Lebesgue	dominated
	convergence theorem	with	\eqref{cv-pen-sol}	and		\eqref{cv-pen-sol--}	ensure	\eqref{semigroup-cv}.\\
    
	From	\eqref{semigroup-epsion},	we	recall	$(P_{t+s}^\epsilon\varphi)(\eta)=(P_{t}^\epsilon	(P_{s}^\epsilon\varphi))(\eta)	\quad	\forall	t,s>0.$	We	show	that	we	can	pass	to	the	limit	as	$\epsilon\to0$.	Indeed,	
	by	using	\eqref{semigroup-cv}	we	get	that	$(P_{t+s}^\epsilon\varphi)(\eta)$		converges	to	$(P_{t+s}\varphi)(\eta)$,
	$(P_{s}^\epsilon\varphi)_\epsilon$	is	uniformly	bounded	and	equicontinuous	on	$H$	thanks	to	\eqref{Feller-penalization}.	Moreover,	by	\eqref{semigroup-cv}		one	has
	$	\displaystyle\lim_{\epsilon	\to	0}(P_s^\epsilon\varphi)(\eta)=(P_s\varphi)(\eta)$.	On	the	other	hand,	denote	by	$(\mu_\epsilon^t)_\epsilon$	the	laws	of	$(u_\epsilon(t;\eta))_\epsilon$	and	$\mu^t$	the	law	of	$u(t;\eta)$	on	$H$ and $K_\psi$ respectively,	then	\eqref{cv-pen-sol}	and		\eqref{cv-pen-sol--}	ensure	that
	$$	\displaystyle\lim_{\epsilon	\to	0}\int_H\phi	d\mu_\epsilon^t=\int_H\phi	d\mu^t,\quad	\forall	\phi\in	\mathcal{C}_b(K_\psi).$$
	Moreover,	recall	that	$\mu^t(K_\psi)=1.$	Next,	\cite[Lemma	1]{Zambotti}	gives
	\begin{align}
		\displaystyle\lim_{\epsilon	\to	0}\int_HP_s^\epsilon\phi	d\mu_\epsilon^t=\int_{K_\psi}P_s\phi	d\mu^t,\quad	\forall	\phi\in	\mathcal{C}_b(K_\psi).
	\end{align}
	Finally, we obtain 
	$$\displaystyle\lim_{\epsilon	\to	0}(P_{t}^\epsilon	(P_{s}^\epsilon\varphi))(\eta)=\displaystyle\lim_{\epsilon\to	0}\int_HP_s^\epsilon\varphi	d\mu_\epsilon^t	=	\int_{K_\psi}P_s\varphi	d\mu^t=(P_{t}	(P_{s}\varphi))(\eta).$$	Thus	$(P_{t+s}\varphi)(\eta)=(P_{t}	(P_{s}\varphi))(\eta).$	The	last	claim	follows	similarly	to	\eqref{semigroup-cv},	by	using	\eqref{cv-pen-sol},		\eqref{cv-pen-sol--}	and	\eqref{stationary-eps}.
	%		\end{itemize}
\end{proof}

\subsection{Proof	of	Theorem	\ref{THm1} and		Theorem	\ref{Thm1-bis}}\label{subsec-existence-invariant}

Let	$\eta\in	K_\psi$	and	denote	by	$\mu_t^{\eta}$	the	law	of	$u(t;\eta),	t\geq	0$.	We		recall	that	$$	P_t\varphi(\eta)=\E[\varphi(u(t;\eta))]=\int_{K_\psi}\varphi(y)\mu_t^{\eta}(dy)=\langle	\varphi,\mu_t^{\eta}\rangle=	\langle	P_t\varphi,\delta_{\eta}\rangle=\langle	\varphi,P_t^*\delta_{\eta}\rangle,	\quad	\forall	\varphi\in	\mathcal{C}_b(K_\psi),	$$	
where $\langle	\cdot, \cdot\rangle$	denotes the duality  between $\mathcal{C}_b(H)$ and	$\mathcal{P}(H)$,	 the space of all Borel probability measures on 	$H$.	Recall	that	$(P_t^*)_t$	is	the adjoint  semigroup defined on	$\mathcal{P}(H)$	by
\begin{align}\label{adjoint-semigroup}
	P_t^*\mu(\Gamma)=\int_HP_t(x,\Gamma)\mu(dx)	\text{	with	}	P_t(\eta,\Gamma):=P(u(t,\eta)\in\Gamma)	\text{	for	any	}	\Gamma\in\mathcal{B}(H).
\end{align}

We		use	\textit{"Krylov-Bogoliubov Theorem"}	(see	\textit{e.g.}	\cite[Theorem 3.1.1]{Daprato-ergodicity})		to construct at least one invariant measure.	Thus,	it	is	sufficient		(see	\textit{e.g.} \cite[Corollary 3.1.2]{Daprato-ergodicity})	to	show	the	following.
\begin{prop}
	The	set	$\{	\nu_n:=\displaystyle\dfrac{1}{t_n}\int_0^{t_n}	\mu_s^\eta	ds;	\hspace*{0.12cm}	n\in	\mathbb{N}^*\}$	is	tight	on	$H$.
\end{prop}
\begin{proof}
Let	$\eta\in	K_\psi$	and	$t\in	[0,T]$,
we distinguish two cases.
\subsubsection*{i. The case $p\leq 2 <+\infty$ }
By	using	It\^o's	formula	to	the	solution	of	\eqref{I-invariant----} with $\mathbf{F}\equiv 0$,	we	get
		\begin{align}
	\dfrac{1}{2}\Vert u(t) \Vert_H^2&+\displaystyle\int_0^t\langle A(u), u\rangle ds=\dfrac{1}{2}\Vert \eta \Vert_H^2+\displaystyle\int_0^t\langle -k, u\rangle ds+\displaystyle\int_0^t\langle f, u\rangle ds\notag\\
	&+\displaystyle\int_0^t\langle G(u)dW(s), u\rangle+\dfrac{1}{2}\displaystyle\int_0^t\Vert G(u)\Vert_{L_2(H_0,H)}^2 ds	\label{key-estimate-1}. 
	\end{align}
		By	using	$H_{2,1}$,	$H_{3,2}$	and	taking	the	expectation,	we	obtain	for	any	$K>0$
		\begin{align*}
		\dfrac{1}{2}\E\Vert u(t) \Vert_H^2+\alpha\displaystyle\int_0^t\E\Vert	u\Vert_V^p ds&\leq	\lambda\int_0^t\E\Vert	u\Vert_H^2 ds+l_1t+\dfrac{1}{2}\Vert \eta \Vert_H^2+\displaystyle\E\int_0^t\vert\langle f-k, u\rangle\vert ds\\&\quad+\dfrac{M(1+K^2)}{2}t+	\dfrac{L_G(1+K^2)}{2K^2}\displaystyle\E\int_0^t\Vert	u\Vert_H^2 ds. 
		\end{align*}
	Let	$0<\delta<1$,
		thanks	to	Lewy-Stampacchia	inequality	\eqref{LS},	$H_5$	and	by	using	Young	inequality,	we	obtain
		\begin{align}
			\displaystyle\E\int_0^t\vert\langle f-k, u\rangle\vert ds&\leq	\displaystyle\E\int_0^t[\Vert k\Vert_{V^\prime}+\Vert f\Vert_{V^\prime}] \Vert	u\Vert_V ds\notag\\
		&\leq	\delta\alpha\displaystyle\int_0^t\E\Vert	u\Vert_V^p ds+\dfrac{C(p,f,\psi)}{(\alpha\delta)^{\frac{1}{p-1}}}t.\notag
		\end{align}
		Thus,	the	following	inequality	holds
			\begin{align}\label{inequality-tight}
		\dfrac{1}{2}\E\Vert u(t) \Vert_H^2+(1-\delta)\alpha\displaystyle\int_0^t\E\Vert	u\Vert_V^p ds\leq	\dfrac{1}{2}\Vert \eta \Vert_H^2&+[\dfrac{C(p,f,\psi)}{(\alpha\delta)^{\frac{1}{p-1}}}+l_1+M]t\notag\\
		&+[\lambda+\dfrac{L_G(1+K^2)}{2K^2}]\int_0^t\E\Vert	u\Vert_H^2 ds.
		\end{align}
		We	distinguish	the	following	cases:
		\begin{enumerate}
			\item	\textbf{The	case	$p>2$:}	since	$V\hookrightarrow	H$	and	thanks	to	Young	inequality	we	get	$$\vert\lambda+L_G\vert\Vert	u\Vert_H^2	\leq	\dfrac{\alpha}{4}\Vert	u\Vert_V^p+C(\alpha,p,\lambda,L_G)$$
			and	\eqref{inequality-tight}	ensures	with	$\delta=\dfrac{1}{4}$
			\begin{align}\label{inequality-tight-1}
		\displaystyle\int_0^t\E\Vert	u\Vert_V^p ds\leq	\dfrac{1}{\alpha}	[\Vert \eta \Vert_H^2+C(f,\psi,l_1,\alpha,p,\lambda,L_G,\delta)t].
		\end{align}
		\item	\textbf{The	case	$p=2$:}		since	$V\hookrightarrow	H$	and		thanks	to	\eqref{const-embeeding},	we	obtain	
		$$	[\lambda+\dfrac{L_G(1+K^2)}{2K^2}]^+\int_0^t\E\Vert	u\Vert_H^2 ds	\leq	C_D[\lambda+\dfrac{L_G(1+K^2)}{2K^2}]^+\int_0^t\E\Vert	u\Vert_V^2 ds,	$$
					where	$v^+=\max(0,v);	v\in	\mathbb{R}$.	Now	\eqref{inequality-tight}	and	\eqref{condition-invariant}	with	$p=2$
			ensure
				\begin{align}\label{inequality-tight-2}
		\displaystyle\int_0^t\E\Vert	u\Vert_V^2 ds\leq	\dfrac{1}{2\bar{\alpha}}	[\Vert \eta \Vert_H^2+C(f,\psi,l_1,\alpha,p,\lambda,L_G,\delta)t];\quad	\forall		\delta	\in	]0,1[.
		\end{align}
%%		\item	\underline{The	case	$\max(1,\frac{2d}{d+2})<p<2:$}		\eqref{inequality-tight}	and		\eqref{condition-invariant}	with	$p<2$	yield
		
%%%		\begin{align}\label{inequality-tight-3}
	%		\displaystyle\int_0^t\E\Vert	u\Vert_V^p ds\leq	\dfrac{1}{2(1-\delta)\alpha}	[\Vert \eta \Vert_H^2+C(f,\psi,l_1,\alpha,p,\lambda,L_G,\delta)t];\quad	\forall	0<\delta<1.
	%%%%		\end{align}
	\item	\textbf{The	case	$p=2$		and	bounded	multiplicative	noise:}		it	follows	from		\eqref{key-estimate-1}	after	using	\eqref{bound-noise-p=2}	and	$\lambda\leq	0$	and	arguments	already	detailed		that
		\begin{align}\label{inequality-tight-4}
		\dfrac{1}{2}\E\Vert u(t) \Vert_H^2+(1-\delta)\alpha\displaystyle\int_0^t\E\Vert	u\Vert_V^p ds\leq	\dfrac{1}{2}\Vert \eta \Vert_H^2&+[\dfrac{C(p,f,\psi)}{(\alpha\delta)^{\frac{1}{p-1}}}+l_1+\dfrac{\mathbf{K}}{2}]t.
		\end{align}
		\end{enumerate}
	Consequently,	
	in	all	the	cases,	we	have	for	$\delta\in]0,1[$
	\begin{align}\label{inequality-tight-result}
	\E\Vert u(t)\Vert_H^2+\displaystyle\int_0^t\E\Vert	u\Vert_V^p ds\leq	C(\alpha)	[\Vert \eta \Vert_H^2+C(f,\psi,l_1,\alpha,p,\lambda,L_G)t].
	\end{align}
	
	Let	$B_R:=\{	v\in	V:	\Vert	v\Vert_V	\leq	R\};	R\in	\mathbb{N}	$	and	note	that	$B_R$	is	relatively compact	in	$H$.	Let	$t_n>0,	R>0$.	By	using	Chebyshev-Markov inequality,	we	write
	\begin{align*}
		\nu_n(B_R^c)=\displaystyle\dfrac{1}{t_n}\int_0^{t_n}	\mu_s^\eta(B_R^c)	ds=\displaystyle\dfrac{1}{t_n}\int_0^{t_n}	P(\Vert	u(s;\eta)\Vert_V>R)	ds	\leq	\displaystyle\dfrac{1}{t_nR^p}\int_0^{t_n}	\E\Vert	u(s;\eta)\Vert_V^p	ds.	
	\end{align*}
Let	$\varrho>0	$,	by	using	\eqref{inequality-tight-result},	we	are	able	to	deduce
	\begin{align}
		\nu_n(B_R^c)	\leq	\displaystyle\dfrac{C(\alpha)}{R^p}	[\dfrac{\Vert \eta \Vert_H^2}{t_n}+C(f,\psi,l_1,\alpha,p,\lambda,L)]\leq	\varrho,
	\end{align}
	provided	$R$	is	chosen	large	enough,	which	ensures		the	existence	of	an	invariant	measure.
    \subsubsection*{ii. The case 	$\underline{\mathbf{p}}<p<2$}
	\begin{enumerate}
	\item		\textbf{The	case	of	Lipschitz	multiplicative	noise:}	it	follows	from	Corollary	\ref{cor-tightness-bis}	the	existence	of		$\mathcal{K}>0$	such	that
	\begin{align}\label{inequality-tight-Bis-}
	\dfrac{1}{\alpha}\E\Vert u(t)\Vert_H^2&+\E\int_0^t \Vert u(s)\Vert_V^pds    \leq	\dfrac{4}{\alpha}[	\|y_0\|^2_{H}+\mathcal{K}(t+1)	],
	\end{align}
for		$\frac{L_G(1+K^2)}{2K^2}+\lambda+L_F\leq	0$	and	any	$t\in[0,T]$ and any $K>0$.
    \item	\textbf{ The	case	of	bounded	multiplicative	noise:}		it	follows	from		\eqref{appendix-1}	after	using	\eqref{bound-noise-p=2}, $\lambda+L_F\leq	0$	and	arguments	already	detailed	in  
Corollary	\ref{cor-tightness-bis}	the	existence	of		$C>0$	such	that
		\begin{align}\label{inequality-tight-4}
\dfrac{1}{\alpha}\E\Vert u(t)\Vert_H^2&+\E\int_0^t \Vert u(s)\Vert_V^pds    \leq	\dfrac{4}{\alpha}[	\|y_0\|^2_{H}+C(t+1)	].
\end{align}		

		\end{enumerate}
Arguments	already	detailed ensure		the	existence	of	an	invariant	measure in this case.
  
\end{proof}
\subsubsection*{Concentration}
Let	$n\in	\mathbb{N}^*$,	
thanks	to	\eqref{inequality-tight-result}		we	obtain\begin{align*}
	\nu_n(\Vert	\cdot\Vert_V^p)=\dfrac{1}{n}\int_0^n\E\Vert	u(s;\eta)\Vert_V^pds\leq	C(\alpha)	[\dfrac{\Vert \eta \Vert_H^2}{n}+C(f,\psi,l_1,\alpha,p,\lambda,L_G)]	\leq	C,	
\end{align*}
where	$C>0$	independent	of	$n$.	Since	$\mu$	is	the	weak	limit	of		a	subsequence	of	$\nu_n$,	it	follows	that	$\mu(\Vert	\cdot\Vert_V^p)\leq	C.$	That	is
$
\mu(\Vert	\cdot\Vert_V^p)=\int_H\Vert	x\Vert_V^p\mu(dx)	\leq	C.
$\\
Since	$\nu_n(K_\psi):=\displaystyle\dfrac{1}{t_n}\int_0^{t_n}	\mu_s^\eta	(K_\psi)ds=1$,	it	follows	that	$\mu(K_\psi)=1.$
%%%%% 26/02/2024
\subsection{Proof of    Proposition \ref{prop-ergodic-full}}
\begin{prop}\label{prop-ergodic-0}
Let $p\geq 2$\footnote{We recall that $\mathbf{F}\equiv 0$ in \eqref{I-invariant----} in this case.},		 there exists $\widetilde{\mathbf{K}}>0$ such that
	\begin{align}\label{prperty-invariant}
\displaystyle\int_H\Vert	x\Vert_V^p\mu(dx)	\leq \widetilde{\mathbf{K}}.	
	\end{align}	
    for any				invariant	measure	 $\mu$ for	the	semigroup	$(P_t)_t$	defined	by	\eqref{transition-prob}.
\end{prop}
\begin{proof}
	Let	$\eta\in	K_\psi$	and		consider		the	following	function		$f_\epsilon:x\mapsto	\dfrac{x}{1+\epsilon	x},	\epsilon>0	$.	Note	that	$f_\epsilon\in		C^2_b(\mathbb{R}_+)$	satisfying
	\begin{align*}
	x\in	\mathbb{R}_+:\quad	f_\epsilon^\prime(x)=\dfrac{1}{(1+\epsilon	x)^2}>0,	\quad	f_\epsilon^{\prime\prime}(x)=-\dfrac{2\epsilon}{(1+\epsilon	x)^3}<0.
	\end{align*}
	Let	$\epsilon>0$,	by	applying	It\^o's	formula	to	the	process	$\Vert	u\Vert_H^2$	given	by	\eqref{key-estimate-1}	and	the	function	$f_\epsilon$,	we	get
	\begin{align}
f_\epsilon(\Vert u(t) \Vert_H^2)-&f_\epsilon(\Vert \eta \Vert_H^2)+2\displaystyle\int_0^tf_\epsilon^\prime(\Vert	u(s)\Vert_H^2)\langle A(u), u\rangle ds=2\displaystyle\int_0^tf_\epsilon^\prime(\Vert	u(s)\Vert_H^2)\langle f-k, u\rangle ds\notag\\
	&+2\displaystyle\int_0^tf_\epsilon^\prime(\Vert	u(s)\Vert_H^2)\langle G(u)dW(s), u\rangle+\displaystyle\int_0^tf_\epsilon^{\prime}(\Vert	u(s)\Vert_H^2)\Vert G(u)\Vert_{L_2(H_0,H)}^2 ds\notag\\
	&+2\displaystyle\int_0^tf_\epsilon^{\prime\prime}(\Vert	u(s)\Vert_H^2)\Vert( G(u),u)\Vert_{L_2(H_0,\mathbb{R})}^2 ds\label{ergodicity-est}.		\end{align}
	Since	$f^{\prime\prime}_\epsilon<0$,	the	last	term	is	non	positive.	On	the	other	hand,
	by	using	that	$f^\prime_\epsilon>0$,			$H_{2,1}$,	$H_{3}$,	we	obtain	for	any	$K>0$
	\begin{align}
	f_\epsilon(\Vert u(t) \Vert_H^2)&+2\alpha\displaystyle\int_0^tf_\epsilon^\prime(\Vert	u(s)\Vert_H^2)\Vert	u\Vert_V^p ds\leq f_\epsilon(\Vert \eta \Vert_H^2)
	+	2\lambda\int_0^tf_\epsilon^\prime(\Vert	u(s)\Vert_H^2)\Vert	u\Vert_H^2 ds+2l_1t\notag\\&+2\displaystyle\int_0^tf_\epsilon^\prime(\Vert	u(s)\Vert_H^2)\vert\langle f-k, u\rangle\vert ds+2\displaystyle\int_0^tf_\epsilon^\prime(\Vert	u(s)\Vert_H^2)\langle G(u)dW(s),u\rangle \notag \\&\quad+M(1+K^2)t+	\dfrac{L_G(1+K^2)}{K^2}\displaystyle\int_0^tf_\epsilon^\prime(\Vert	u(s)\Vert_H^2)\Vert	u\Vert_H^2 ds\label{ineq-ergod}. 
	\end{align}
Recall	that	$f^\prime_\epsilon\leq	1$,	which	ensures	that		the	stochastic	integral	is		$(\mathcal{F}_t)$-martingale.	Hence,	by	taking	the	expectation	and	using arguments already detailed,	one	has
	\begin{align*}
&\E	f_\epsilon(\Vert u(t) \Vert_H^2)+2\alpha\E\displaystyle\int_0^tf_\epsilon^\prime(\Vert	u(s)\Vert_H^2)\Vert	u\Vert_V^p ds\leq f_\epsilon(\Vert \eta \Vert_H^2)
	+	2\lambda\E\int_0^tf_\epsilon^\prime(\Vert	u(s)\Vert_H^2)\Vert	u\Vert_H^2 ds+2l_1t\\&+2\displaystyle\E\int_0^tf_\epsilon^\prime(\Vert	u(s)\Vert_H^2)\vert\langle f-k, u\rangle\vert ds+M(1+K^2)t+	\dfrac{L_G(1+K^2)}{K^2}\displaystyle\E\int_0^t f_\epsilon^\prime(\Vert	u(s)\Vert_H^2)\Vert	u\Vert_H^2 ds. 
	\end{align*}
	Let	$0<\delta<1$,
	thanks	to	Lewy-Stampacchia	inequality	\eqref{LS},	$H_5$	and	by	using	Young	inequality,	we	obtain
	\begin{align}
	\displaystyle\E\int_0^tf_\epsilon^\prime(\Vert	u(s)\Vert_H^2)\vert\langle f-k, u\rangle\vert ds&\leq 	\delta\alpha\displaystyle\int_0^t\E f_\epsilon^\prime(\Vert	u(s)\Vert_H^2)\Vert	u\Vert_V^p ds+\dfrac{C(p,f,\psi)}{(\alpha\delta)^{\frac{1}{p-1}}}t.\notag
	\end{align}
	Thus,	the	following	inequality	holds
 	\begin{align}\label{ergo-esti-key}
	\E	f_\epsilon(\Vert u(t) \Vert_H^2)&+2(1-\delta)\alpha	\displaystyle\int_0^t\E f_\epsilon^\prime(\Vert	u(s)\Vert_H^2)\Vert	u\Vert_V^p ds\notag\\&\leq	f_\epsilon(\Vert \eta \Vert_H^2)+\mathbf{Y}t
	+2\mathbf{R}\int_0^t\E f_\epsilon^\prime(\Vert	u(s)\Vert_2^2)\Vert	u\Vert_H^2 ds,
	\end{align}
	where	$\mathbf{Y}=[2\dfrac{C(p,f,\psi)}{(\alpha\delta)^{\frac{1}{p-1}}}+2l_1+M(1+K^2)], \mathbf{R}=\lambda+\dfrac{L_G(1+K^2)}{2K^2}$.
	%%%%%%%%%%%%%
		\begin{enumerate}
		\item	\textbf{The	case	$p>2$:}	 choose $\delta=\dfrac{1}{2}$ in \eqref{ergo-esti-key}, 
		since	$V\hookrightarrow	H$	and	thanks	to	Young	inequality	we	get	$2\mathbf{R}\Vert	u\Vert_H^2	\leq	\dfrac{\alpha}{2}\Vert	u\Vert_V^p+C(\alpha,\mathbf{R})$. Thus
		\begin{align*}
		&\E	f_\epsilon(\Vert u(t) \Vert_H^2)+\dfrac{\alpha}{2}	\displaystyle\int_0^t\E f_\epsilon^\prime(\Vert	u(s)\Vert_H^2)\Vert	u\Vert_V^p ds\leq	f_\epsilon(\Vert \eta \Vert_H^2)+(\mathbf{Y}+
		C(\alpha,\mathbf{R}))t.
		\end{align*}
	\item	\textbf{The	case	$p=2$:}		since	$V\hookrightarrow	H$	and		thanks	to	\eqref{const-embeeding},	we	obtain	
		$$	2[\lambda+\dfrac{L_G(1+K^2)}{2K^2}]^+\int_0^t\E\Vert	u\Vert_H^2 ds	\leq	2C_D[\lambda+\dfrac{L_G(1+K^2)}{2K^2}]^+\int_0^t\E\Vert	u\Vert_V^2 ds,	$$
		where	$v^+=\max(0,v);	v\in	\mathbb{R}$.	Now	\eqref{ergo-esti-key}	and	\eqref{condition-invariant}	with	$p=2$
		ensure
		\begin{align*}
		&\E	f_\epsilon(\Vert u(t) \Vert_H^2)+2\bar{\alpha}	\displaystyle\int_0^t\E f_\epsilon^\prime(\Vert	u(s)\Vert_H^2)\Vert	u\Vert_V^2 ds\leq	f_\epsilon(\Vert \eta \Vert_H^2)+\mathbf{Y}t.
		\end{align*}
		%%		\item	\underline{The	case	$\max(1,\frac{2d}{d+2})<p<2:$}		\eqref{inequality-tight}	and		\eqref{condition-invariant}	with	$p<2$	yield
		
		%%%		\begin{align}\label{inequality-tight-3}
		%		\displaystyle\int_0^t\E\Vert	u\Vert_V^p ds\leq	\dfrac{1}{2(1-\delta)\alpha}	[\Vert \eta \Vert_H^2+C(f,\psi,l_1,\alpha,p,\lambda,L_G,\delta)t];\quad	\forall	0<\delta<1.
		%%%%		\end{align}
		\item	\textbf{The	case	$p=2$		and	bounded	multiplicative	noise:}		it	follows	from		\eqref{ergodicity-est}	after	using	\eqref{bound-noise-p=2}		and	arguments	already	detailed		that
		\begin{align*}
	\E	f_\epsilon(\Vert u(t) \Vert_H^2)+2(1-\delta)\alpha\displaystyle\int_0^t\E f_\epsilon^\prime(\Vert	u(s)\Vert_H^2)\Vert	u\Vert_V^2 ds\leq	f_\epsilon(\Vert \eta \Vert_H^2)&+[2\dfrac{C(p,f,\psi)}{(\alpha\delta)^{\frac{1}{p-1}}}+2l_1+\mathbf{K}]t.
		\end{align*}
        \end{enumerate}
		Therefore,
	in	all	the	cases,	we	have	for	$\delta\in]0,1[$, there exist $\mathbf{C}:=\mathbf{Y}+\mathbf{K}
	+C(\alpha,\mathbf{R})$ and $c_{\alpha,\bar{\alpha}}>0$ such that
\begin{align*}
&\E	f_\epsilon(\Vert u(t) \Vert_H^2)+c_{\alpha,\bar{\alpha}}	\displaystyle\int_0^t\E f_\epsilon^\prime(\Vert	u(s)\Vert_H^2)\Vert	u\Vert_V^p ds\leq	f_\epsilon(\Vert \eta \Vert_H^2)+\mathbf{C}t,\quad  p\geq \max(1,\frac{2d}{d+2}).
\end{align*}
 Let $p\geq 2$ and recall that 
 $V\hookrightarrow	H$, thus  $\dfrac{1}{C_D}\Vert	u\Vert_H^2\leq \Vert	u\Vert_V^2\leq \Vert	u\Vert_V^p+C_p.$
	%%%%%%
		Therefore
	\begin{align*}%%\label{est-invariant2}
	\E	f_\epsilon(\Vert	u(t)\Vert_H^2)+\dfrac{c_{\alpha,\bar{\alpha}}}{C_D} \int_0^t\E	\dfrac{\Vert	u(s)\Vert_H^2}{(1+\epsilon	\Vert	u(s)\Vert_H^2)^2}	ds	\leq	f_\epsilon(\Vert	\eta\Vert_H^2)
	+(\mathbf{C}+C_pc_{\alpha,\bar{\alpha}})t,\quad	\epsilon>0.	\end{align*}
	For	$y\in	H$,	set	$F_\epsilon(y)=f_\epsilon\circ	\Vert	y\Vert_H^2	$	and	note	that	$F_\epsilon\in	C_b(H)$.	We	recall	that	$P_tF_\epsilon(\eta)=\E	F_\epsilon(u(t))$.	Let	$\mu$	be	an	invariant	measure	for	$(P_t)_t$,	by	the	definition	of	invariant	measure	for	the	semigroup	$(P_t)_t$,	we	obtain	after	integrating	with	respect	to	$\mu$
	\begin{align*}%%\label{est-invariant2}
	\dfrac{c_{\alpha,\bar{\alpha}}}{C_D} \int_H\int_0^t\E	\dfrac{\Vert	u(s)\Vert_H^2}{(1+\epsilon	\Vert	u(s)\Vert_H^2)^2}	ds	d\mu	\leq	
	(\mathbf{C}+C_pc_{\alpha,\bar{\alpha}})t.	\end{align*}
	Let	$g_\epsilon(x)=\dfrac{x}{(1+\epsilon	x)^2}$	and	$G_\epsilon=g_\epsilon\circ\Vert	\cdot\Vert_H^2\in	C_b(H)$.	Hence	$\E	\dfrac{\Vert	u(s)\Vert_H^2}{(1+\epsilon	\Vert	u(s)\Vert_H^2)^2}:=P_sG_\epsilon(\eta).$
	Tonelli	theorem	and	the	invariance	of	$\mu$	ensure
	\begin{align*}%%\label{est-invariant2}
	\dfrac{c_{\alpha,\bar{\alpha}}}{C_D} \int_H\int_0^t\E	\dfrac{\Vert	u(s)\Vert_H^2}{(1+\epsilon	\Vert	u(s)\Vert_H^2)^2}	ds	d\mu&=\dfrac{c_{\alpha,\bar{\alpha}}}{C_D} \int_0^t\int_H	P_sG_\epsilon(\eta)d\mu	ds\\&=t\dfrac{c_{\alpha,\bar{\alpha}}}{C_D} \int_H	G_\epsilon(\eta)d\mu		\leq	
	(\mathbf{C}+C_pc_{\alpha,\bar{\alpha}})t.	\end{align*}
	Finally,	by	letting	$\epsilon\to	0$	and	using	monotone	convergence	theorem	we	get
	
	\begin{align}\label{prperty-invariant-1--0}%%\label{est-invariant2}
	\int_H	\Vert	y\Vert_H^2\mu(dy)		\leq	
	\dfrac{(\mathbf{C}+C_pc_{\alpha,\bar{\alpha}})C_D}{c_{\alpha,\bar{\alpha}}}\leq\widetilde{\mathbf{K}}_1.	\end{align}

	Next,	we	use	the	last	inequality	\eqref{prperty-invariant-1--0}	to	show
		\eqref{prperty-invariant}.	
        	Define	the	following	non	decreasing	sequence	
	\begin{align}\label{sequ-appro-invari}
	n\in	\mathbb{N}:\quad	F_n:H\to	\mathbb{R}_+\cup\{+\infty\};\quad	u\mapsto		\begin{cases}	\Vert	u\Vert_V^p	&\text{	if	}	\Vert	u\Vert_V\leq	n;
	\\[0.1cm]
	n^p	&\text{	else.		}
	\end{cases}
	\end{align}
	Note	that	$F_n$	converges	to	$F_V:=\displaystyle\sup_{n}F_n$,	where	
	\begin{align*}
	F_V:H\to	\mathbb{R}_+\cup\{+\infty\};\quad	u\mapsto		\begin{cases}	\Vert	u\Vert_V^p	\text{	if	}		u\in	V;
	\\[0.1cm]	
	+\infty	\text{	if	}	u\in	H\setminus	V.
	\end{cases}
	\end{align*}	
	%%%$F:H\to	\mathbb{R}_+\cup\{+\infty\}$	such	that	$F(u)=\Vert	%%%u\Vert_X^41_X(u)+\infty\cdot1_{H/	X}$.
	It	is	clear		that	$F_n\in	B_b(H)$\footnote{$B_b(H)$	denotes	the	set	of	bounded	Borel	functions	on	$H$.}	for	every	$n\in	\mathbb{N}$	and	$	F_n(u)\leq	\Vert	u\Vert_V^p$.	By	using	the	invariance	of	$\mu$,	we	are	able	to	infer
	\begin{align*}
	\int_HF_nd\mu =\int_0^1\int_HP_sF_nd\mu ds=\int_0^1\int_H\E	F_n(u(s))d\mu	ds=\int_H\int_0^1\E	F_n(u(s))	dsd\mu.
	\end{align*}
By using 	\eqref{inequality-tight-result}, one has
	\begin{align}\label{inequality-tight-result-}
\displaystyle\int_0^T\E	F_n(u(s))	ds	\leq	\displaystyle\int_0^T\E\Vert	u\Vert_V^p ds\leq	C(\alpha)	[\Vert \eta \Vert_H^2+C(f,\psi,l_1,\alpha,p,\lambda,L_G)T].
	\end{align}
	Set	$T=1$	and		integrate	with	respect	to	$\mu$	the	last	inequality,	one	has
	\begin{align*}
	\int_HF_n(\eta)		d\mu=\int_H\int_0^1P_sF_n(\eta)	ds	d\mu&=\int_H\int_0^1\E	F_n(u(s))	ds	d\mu	\\& \leq	C(\alpha)	[\int_H\Vert \eta \Vert_H^2d\mu+C(f,\psi,l_1,\alpha,p,\lambda,L_G)]	.
	\end{align*}
	Consequently,	the	monotone	convergence	theorem		and	\eqref{prperty-invariant-1--0}	give
	\begin{align*}
\displaystyle\int_H\Vert	x\Vert_V^p\mu(dx)\leq	\int_HF_V(\eta)		d\mu	\leq C(\alpha)	[\widetilde{\mathbf{K}}_1+C(f,\psi,l_1,\alpha,p,\lambda,L_G)]]\leq\widetilde{\mathbf{K}}.
	\end{align*}
    \end{proof}
    \begin{prop}(Ergodicity)\label{prop-ergodic}
    Let $p\geq 2$,	then there exists at least one ergodic invariant measure  for	the	semigroup	$(P_t)_t$	defined	by	\eqref{transition-prob}.	
\end{prop}
\begin{proof}
      Denote	by	$	\Lambda$,		the	set	of	all	invariant	measures	for	the	Markov	semigroup	$(P_t)_t$	defined	by		\eqref{transition-prob}.		From	Subsection	\ref{subsec-existence-invariant},
$\Lambda$	is	nonempty	convex	subset	of	$(C_b(H))^\prime$	and	\eqref{prperty-invariant}		ensures	that	$\Lambda$	is	tight,	since	$2\leq p<\infty$.	Therefore,	Krein-Milman theorem	(see	\textit{e.g.}	\cite[Theorem	1.13]{Brezis})	ensures	that		the	set	of	extreme points	is	non	empty	and	then	any	 extremal point of $\Lambda$ is an ergodic 	invariant	measure,	since
	the	set	of	all	invariant	ergodic	measures	of	$(P_t)_t$	coincides	with	the	set	of	all	extremal	points	of	$\Lambda$,	see	\cite[Theorem	5.18]{Daprato-lecture}. 
\end{proof}
Concerning the case $\underline{\mathbf{p}} < p<2$, we prove only  that the $p^{th}$-moment is bounded.
\begin{prop}\label{ergo-1-prop}
         Let $\underline{\mathbf{p}} < p<2$, there exists  $\bar{\mathbf{K}}>0$ such that 
    \begin{align}\label{prperty-invariant-2}
\displaystyle\int_H\Vert	x\Vert_H^p\mu(dx)	\leq \bar{\mathbf{K}}.	
	\end{align}	
  for    any invariant measure $\mu$ for	the	semigroup	$(P_t)_t$	defined	by	\eqref{transition-prob}.
\end{prop}
\begin{proof}
                Let  $\underline{\mathbf{p}}<p<2$, we have
        \begin{enumerate}
        	\item	\textbf{The	case	 of 	bounded	multiplicative	noise:}		if	$\lambda+L_F\leq	0$, it follows from 
Corollary	\ref{cor-tightness-bis}	the	existence	of		$S>0$	such	that
 \begin{align*}
&\E	f_\epsilon(\Vert u(t)-\psi \Vert_H^2)+\dfrac{\alpha}{2}\displaystyle\E	\int_0^tf_\epsilon^\prime(\Vert	u(s)-\psi\Vert_H^2)\Vert u(s)\Vert_V^pds\leq f_\epsilon(\Vert u_0 -\psi\Vert_H^2)+(S+\mathbf{K})t.\end{align*}
	\item	   \textbf{The case of Lipschitz noise with	$\dfrac{L_G(1+K^2)}{2K^2}+\lambda+L_F\leq 0:	$}	it	follows	from	Corollary	\ref{cor-tightness-bis}	the	existence	of		$S>0$	such	that
         		 \begin{align*}
&\E	f_\epsilon(\Vert u(t)-\psi \Vert_H^2)+\dfrac{\alpha}{2}\displaystyle\E	\int_0^tf_\epsilon^\prime(\Vert	u(s)-\psi\Vert_H^2)\Vert u(s)\Vert_V^pds\leq f_\epsilon(\Vert u_0 -\psi\Vert_H^2)+St. \end{align*}  
	\end{enumerate}
Let $\eta \in K_\psi$,  there exists $\mathbf{C}>0$ in both cases such that, after using  
 $V\hookrightarrow	H$,  one has
	\begin{align*}%%\label{est-invariant2}
	2\E	f_\epsilon(\Vert	u(t)-\psi\Vert_H^2)+\dfrac{\alpha}{C_D^{p/2}} \int_0^t\E	\dfrac{\Vert	u(s)\Vert_H^p}{(1+\epsilon	\Vert	u(s)-\psi\Vert_H^2)^2}	ds	\leq	2f_\epsilon(\Vert	\eta-\psi\Vert_H^2)
	+\mathbf{C}t,\quad	\epsilon>0.	\end{align*}
   Since $\psi \in V$, by repeating the same arguments as above, namely the definition of invariant measure one has
	\begin{align*}%%\label{est-invariant2}
	\dfrac{\alpha}{C_D^{p/2}} \int_H\int_0^t\E	\dfrac{\Vert	u(s)\Vert_H^p}{(1+\epsilon	\Vert	u(s)-\psi\Vert_H^2)^2}	ds	d\mu	\leq	
	\mathbf{C}t.	\end{align*}
	Let		$\widetilde{G}_\epsilon(\xi)= \dfrac{\Vert	\xi\Vert_H^p}{(1+\epsilon	\Vert	\xi-\psi\Vert_H^2)^2}\in	C_b(H), \xi \in H$.	Hence	$\E	\dfrac{\Vert	u(s)\Vert_H^p}{(1+\epsilon	\Vert	u(s)-\psi\Vert_H^2)^2}:=P_s\widetilde{G}_\epsilon(\eta).$
    Let	$\mu$	be	an	invariant	measure	for	$(P_t)_t.$
	Tonelli	theorem	and	the	invariance	of	$\mu$	ensure
	\begin{align*}%%\label{est-invariant2}
	 \int_H\int_0^t\E	\dfrac{\Vert	u(s)\Vert_H^p}{(1+\epsilon	\Vert	u(s)-\psi\Vert_H^2)^p}	ds	d\mu&= \int_0^t\int_H	P_s\widetilde{G}_\epsilon(\eta)d\mu	ds=t \int_H	\widetilde{G}_\epsilon(\eta)d\mu		\leq	
	\dfrac{C_D^{p/2}}{\alpha}\mathbf{C}t.	\end{align*}
	Finally,	by	letting	$\epsilon\to	0$	and	using	monotone	convergence	theorem	we	get
	\begin{align}\label{prperty-invariant-1--}%%\label{est-invariant2}
	\int_H	\Vert	y\Vert_H^p\mu(dy)		\leq	
	\dfrac{C_D^{p/2}}{\alpha}\mathbf{C}\leq\widetilde{\mathbf{K}}_1.	\end{align}
\end{proof}
\begin{remark}Let  $\underline{\mathbf{p}} < p<2$,
 from \eqref{prperty-invariant-2}, we have  
 $\int_H	\Vert	y\Vert_H^{p}\mu(dy)	$ is uniformly bounded for any invariant measure  $\mu \in \Lambda$. In order to prove that the set of invariant measures is tight  in the same way as \eqref{inequality-tight-result-}  we need   $
	(\int_H	\Vert	y\Vert_H^{2}\mu(dy)	)_{\mu \in \Lambda}$ is bounded  which is not  satisfied. 
\end{remark}

%%%

\subsection{Proof	of	Theorem	\ref{THm2} and Theorem \ref{THm2-2}}\label{Subsection-uniq-1}
Let us prove  the uniqueness of invariant measures. It is well known
that one only needs to show the uniqueness of ergodic invariant measures,	see	\textit{e.g.}	\cite[Theorem	5.16]{Daprato-lecture}.	For	that,	let	$\mu_1$	and	$\mu_2$	be	two	 ergodic invariant measures,	then	for any bounded Lipschitz
function $\varphi$
\begin{align*}
\displaystyle\lim_{T\to	+\infty}\dfrac{1}{T}\int_0^TP_t\varphi	dt=\int_H\varphi(x)	\mu_1(dx)	\text{	in	}	L^2(H,\mu_1),\\
\displaystyle\lim_{T\to	+\infty}\dfrac{1}{T}\int_0^TP_t\varphi	dt=\int_H\varphi(y)	\mu_2(dy)	\text{	in	}	L^2(H,\mu_2).
\end{align*}
Hence
\begin{align}
	\Big\vert\dfrac{1}{T}\int_0^TP_t\varphi(x)	dt-\dfrac{1}{T}\int_0^TP_t\varphi(y)	dt\Big\vert	&\leq	\dfrac{1}{T}\int_0^T	\Big\vert	P_t\varphi(x)	-P_t\varphi(y)\Big\vert	dt\notag\\
	&\leq		\dfrac{\Vert	\varphi\Vert_{Lip}}{T}\int_0^T		\E\Vert	u(t;x)-u(t;y)\Vert_H		dt.\label{eqn-uniq-1}
\end{align}
On	the	other	hand,	let	$(u_1,k_1)$	and	$(u_2,k_2)$	be two solutions to \eqref{I-invariant----} with $\mathbf{F}\equiv 0$, corresponding to two initial data	$x$	and	$y$	respectively.	P-a.s,	
for any $t\in[0,T]$ we have 
\begin{align*}
 u_1(t)-u_2(t)+\displaystyle\int_0^t[k_1-k_2] ds+\displaystyle\int_0^t [A(u_1)-A(u_2)]ds
= \displaystyle\int_0^t [G(u_1)-G(u_2)]dW(s).
\end{align*}
Applying Ito's formula with $F(t,v)=\dfrac{1}{2}\Vert v\Vert_H^2$, one gets for any $t\in [0,T]$
\begin{align*}
&\dfrac{1}{2}\Vert (u_1-u_2)(t)\Vert_H^2+\displaystyle\int_0^t\langle A(u_1)-A(u_2),u_1-u_2\rangle ds+\displaystyle\overbrace{\int_0^t\langle  k_1-k_2, u_1-u_2 \rangle ds}^{\geq	0	\text{,	see	Definition	\ref{def-1}	}}\\
&=\dfrac{1}{2}\Vert x-y\Vert_H^2+\displaystyle\int_0^t\langle [G(u_1)-G(u_2)]dW(s), u_1-u_2\rangle + \dfrac{1}{2}\displaystyle\int_0^t \Vert  G(u_1)-G(u_2) \Vert_{L_2(H_0,H)}^2 ds.
\end{align*}
 Since $\lambda_T Id+A$ is T-monotone	(see	$H_{2,2}$), we	have $$\displaystyle\int_0^t\langle A(u_1)-A(u_2),u_1-u_1\rangle ds \geq -\lambda_T\int_0^t\Vert u_1-u_2\Vert_H^2 ds.$$
We distinguish two cases.
\begin{enumerate}
    \item \textbf{The case of Lipschitzian noise:}  by	taking	the	expectation	and	using		$H_{3,1}$		we	get
 \begin{align*}
\E\Vert (u_1-u_2)(t)\Vert_H^2
\leq	\Vert x-y\Vert_H^2+2[\dfrac{L_G}{2}+\lambda_T]\displaystyle\int_0^t \E\Vert  u_1-u_2 \Vert_{H}^2 ds.
\end{align*}
Therefore
$\E\Vert (u_1-u_2)(t)\Vert_H^2 \leq \Vert x-y\Vert_H^2	e^{2[\frac{L_G}{2}+\lambda_T]t},
$
which	gives
\begin{equation}\label{EQ1-uniquenes}
\E\Vert (u_1-u_2)(t)\Vert_H \leq \Vert x-y\Vert_H	e^{[\frac{L_G}{2}+\lambda_T]t}. 
\end{equation} 
Hence	\eqref{eqn-uniq-1}	and	\eqref{EQ1-uniquenes}	give
\begin{align*}
\Big\vert\dfrac{1}{T}\int_0^TP_t\varphi(x)	dt-\dfrac{1}{T}\int_0^TP_t\varphi(y)	dt\Big\vert	&\leq		\dfrac{\Vert	\varphi\Vert_{Lip}}{T}\int_0^T		\E\Vert	u(t;x)-u(t;y)\Vert_H		dt\\
&	\leq	\Vert x-y\Vert_H	g(T),
\end{align*}
where	$g(T):=\frac{\Vert	\varphi\Vert_{Lip}}{T}\dfrac{2}{L_G+2\lambda_T}(	e^{[\frac{L_G+2\lambda_T}{2}]T}-1)	\to	0$		as	$T\uparrow	+\infty$,
provided	that	$\dfrac{L_G}{2}+\lambda_T< 0$.	
\item \textbf{The case of Lipschitzian and bounded noise:}
if $G$ satisfies \eqref{bound-noise-p=2}, it is not difficult to see that we get
\begin{align*}
\E\Vert (u_1-u_2)(t)\Vert_H^2
\leq	\Vert x-y\Vert_H^2+2\lambda_T\displaystyle\int_0^t \E\Vert  u_1-u_2 \Vert_{H}^2 ds+\mathbf{K}t.
\end{align*}
Thus $
\E\Vert (u_1-u_2)(t)\Vert_H^2
\leq	(\Vert x-y\Vert_H^2+\mathbf{K}t) e^{2\lambda_T t}\leq	(\Vert x-y\Vert_H+\sqrt{\mathbf{K}t})^2 e^{2\lambda_T t}.$
Hence	
\begin{align*}
&\Big\vert\dfrac{1}{T}\int_0^TP_t\varphi(x)	dt-\dfrac{1}{T}\int_0^TP_t\varphi(y)	dt\Big\vert	\leq		\dfrac{\Vert	\varphi\Vert_{Lip}}{T}\int_0^T		\E\Vert	u(t;x)-u(t;y)\Vert_H		dt\\
&	\leq	\dfrac{\Vert	\varphi\Vert_{Lip}}{T}\int_0^T (\Vert x-y\Vert_H+\sqrt{\mathbf{K}t}) e^{\lambda_T t} dt\\
&	\leq	 \dfrac{\Vert	\varphi\Vert_{Lip}\Vert x-y\Vert_H}{\lambda_TT}(e^{\lambda_T T}-1)+\dfrac{\Vert	\varphi\Vert_{Lip}\sqrt{\mathbf{K}}}{T}\int_0^{+\infty } \sqrt{t}e^{\lambda_T t} dt:=H(T).
\end{align*}
A standard calculation ensures that $\int_0^{+\infty } \sqrt{t}e^{\lambda_T t} dt< +\infty$ and $H(T)\to	0$		as	$T\uparrow	+\infty$.\\
\end{enumerate}
Therefore,	the	uniqueness	holds,	namely
$$	\int_H\varphi	d\mu_1=\int_H\varphi	d\mu_2	\text{	for any bounded Lipschitz function }	\varphi	\text{ on	} H. $$
Finally,	let	$\mu$	be	an	invariant	measure	for	$(P_t)$,	for	$\varphi\in	\mathcal{C}_b^1(H)$	we	have
\begin{align}
&\Big	\vert	P_t\varphi(x)-\int_H\varphi(y)\mu(dy)\Big\vert \notag\\
&\leq	
\Big	\vert	\int_H[P_t\varphi(x)-P_t\varphi(y)]\mu(dy)
\leq		\Vert	\varphi\Vert_{1,\infty}\int_H\E\Vert	u(t;x)-u(t;y)\Vert\mu(dy)\notag\\
&\leq	\Vert	\varphi\Vert_{1,\infty}	e^{[\frac{L_G}{2}+\lambda_T]t}\int_H\Vert	x-y\Vert\mu(dy)\leq	\mathcal{G}\Vert	\varphi\Vert_{1,\infty}	e^{[\frac{L_G}{2}+\lambda_T]t},\label{equilibrium-1}
\end{align}	
where $\mathcal{G}:=\int_H\Vert	x-y\Vert\mu(dy)<+\infty$, thanks to \eqref{prperty-invariant-1--0} and  \eqref{prperty-invariant-1--}. If moreover $G$ satisfies \eqref{bound-noise-p=2}, we have instead 
\begin{align}
\Big	\vert	P_t\varphi(x)-\int_H\varphi(y)\mu(dy)\Big\vert\leq	\Vert	\varphi\Vert_{1,\infty}	(\mathcal{G}+\sqrt{\mathbf{K}t}) e^{\lambda_T t}\label{equilibrium-2}
\end{align}	
Moreover,	since		 $\mathcal{C}_b^1(H)$ is dense in $L^2(H,\mu)$, it follows that for any
$\varphi	\in	L^2(H,\mu)$,	one	has	
$$\lim_{t\to	+\infty}P_t\varphi(x)=\int_H\varphi	d\mu,\quad	x\in	H,$$
provided	that	$\dfrac{L_G}{2}+\lambda_T< 0$  and  $\lambda_T< 0$ in \eqref{equilibrium-1}  and \eqref{equilibrium-2} respectively.
Therefore, $\mu$ is ergodic and strongly mixing.
\begin{remark}	If $\underline{\mathbf{p}}<p<2$,	one obtains similarly	$\lambda_T+L_F$	instead	of	$\lambda_T$	in	\eqref{equilibrium-1} and \eqref{equilibrium-2}.	
\end{remark}
\subsection*{Double obstacles problem}
We are interested in this subsection to say a few words about the invariant measures associated with a double obstacle problem \textit{i.e.} 
find $(u,k)$ in a space defined straight after, solution to
\begin{align*}
 \begin{cases}
 du+A(u)ds+ kds= fds+G(u)dW \quad & in \quad D\times  \Omega_T, \\ 
 k=k_2-k_1, \quad \langle k_1,u-\psi_1\rangle= 0, \langle k_2,u-\psi_2\rangle= 0,  \hspace*{0.2cm} k_1, k_2\in (V^\prime)^+  & in \quad  \Omega_T,\\
  \psi_2 \geq u \geq \psi_1 & in \quad  D\times  \Omega_T, \\  
   u=  0   &on \quad\partial D\times \Omega_T, \\
  u(t=0)=u_0 & \text{in } H,\  a.s.
   \end{cases}
\end{align*}
Let us precise the assumptions on the obstacles. Namely, we replace  H$_4$-H$_6$ by the following:
\begin{itemize}
   \item[H$_4^*$] :	The	obstacles $\psi_i\in V$, $\psi_1\leq \psi_2$ 	such	that	$G(\psi_i)=0$ for $i=1,2$.
	\item[H$_5^*$] : 	The	external	force $f\in V^\prime$ such  that 
	$
	 f  - A(\psi_i)=h_i= h_i^+-h_i^-  \in V^*$ for $i=1,2$.
	\item[H$_6^*$] : The	initial	datum	 $u_0 \in H$,	and	 satisfies the constraints, \textit{i.e.} $\psi_1\leq u_0 \leq \psi_2$.
\end{itemize}
Denote by $K$ the convex  set of admissible functions
$$K_{\psi_1}^{\psi_2}=\{v \in L^p(\Omega_T,V), \quad\psi_2(x) \geq     v(x,t,\omega)\geq \psi_1(x) \quad \text{a.e. in } D\times\Omega_T\}.$$
Following \cite[Theorem 3 ]{YV22}, we have
\begin{theorem}\label{THm-double} Assume that H$_1$-H$_3$ and H$_4^*$-H$_6^*$ hold, assume moreover that $h^-_1,h^+_2$ are non negative elements of $L^{\tilde{q}^\prime}(\Omega_T, L^{\tilde{q}^\prime}(D))$\footnote{$\tilde{q}=\min(p,2)$.}. Then, 
		there exists  a unique predictable stochastic process $(u,\rho_1,\rho_2)\in L^p(\Omega_T,V) \times L^{\tilde{q}^\prime}(\Omega_T,L^{\tilde{q}^\prime}( D))\times L^{\tilde{q}^\prime}(\Omega_T,L^{\tilde{q}^\prime}( D))$
		such that: 
	\begin{itemize}[label=$\bullet$]
		\item[i.] $u \in  L^2(\Omega,\mathcal{C}([0,T],H)) \cap K_{\psi_1}^{\psi_2}$, $ u(0)=u_0$.
		\item[ii.] $-\rho_1,\rho_2\geq 0$ and $\forall v \in K_{\psi_1}^{\psi_2}$,  $ \langle \rho_i,u-v\rangle  \geq 0, \quad i=1,2$ a.e. in $\Omega_T$. 
		\item[iii.]  P-a.s, for all $t\in[0,T]$,  		$$u(t)+\displaystyle\int_0^t(\rho_1+\rho_2)ds+\displaystyle\int_0^t A(u,\cdot)ds=u_0+\displaystyle\int_0^t G(u,\cdot)dW(s)+\displaystyle\int_0^t fds.$$ 
		\item[iv.]The following  Lewy-Stampacchia's inequality holds: 
			\begin{align*}
			-h_2^+&=  -\left(f  - A(\psi_2)\right) ^+
			\leq  \partial_t ( u-\displaystyle\int_0^\cdot G(u)dW)+ A(u)-f 
			\leq h_1^-=  \left(f   - A(\psi_1)\right) ^-.
			\end{align*}
	\end{itemize}
\end{theorem} It is not difficult to see that the extension of the results to the case of double obstacles is straightforward, it follows by repeating the same argument of Section \ref{Sec-Proof} and using Theorem \ref{THm-double}.

\section{Appendix:	on		well-posedness	of	\eqref{I-invariant----}}\label{sec-appendix}
Let	$1<p<2$	and		denote	by		$V=W^{1,p}_0(D)\cap	L^2(D),		V^\prime=W^{-1,p^\prime}(D)+L^2(D)$.
	\begin{theorem}\label{TH1perturbation}	Assume	that	H$_1$-H$_7$	hold		and	 $\widetilde{h}^-\in	L^{p^\prime}(D)$	in	\eqref{new-decompsition},  there exists a unique $(u,\rho)\in L^p(\Omega_T;V) \times L^{p^\prime}(\Omega_T;L^{p^\prime}(D))$,	solution	to	\eqref{I-invariant----}, both predictable,  satisfying: 
		\begin{itemize}
			\item  $u\in  L^2(\Omega;C([0,T],H)) \cap K$ and 	$\rho \leq 0$.%%$\rho\in	L^{p^\prime}(\Omega_T\times D)$	with
			\item P-a.s.	for any $t\in[0,T]$: $ u(t)+\displaystyle\int_0^t(A(u)+\mathbf{F}(u)+\rho-f)ds=u_0+\displaystyle\int_0^t G(u)dW(s).$
			\item $	\langle \rho, u-\psi\rangle=0$ a.e. in $\Omega_T$  and,  for	all	$	v\in K$,  $\langle \rho, u-v\rangle\geq 0$ a.e. in $\Omega_T$.
		\end{itemize}
		Moreover,	we	have
	$		\E\displaystyle\sup_{t\in[0,T]}\Vert (u_\epsilon-u)(t)\Vert_H^2\leq   C\epsilon^{\frac{1}{p-1}}$		for	all	$\epsilon>0,
	$
		where	$u_\epsilon$	is	the	unique	solution	of	\eqref{1}.
\end{theorem} 
\begin{remark}\label{Rmq-uniq-perturbation}
	Let	$t\in	[0,T]$	and	$y_1$	and	$y_2$	be	two	solution	to	\eqref{I-invariant----}.	Note	that
	\begin{align*}
	\vert	\int_0^t(\mathbf{F}(y_1)-\mathbf{F}(y_2),y_1-y_2)ds\vert\leq	L_F\int_0^t\Vert	y_1-y_2\Vert_H^2ds,
	\end{align*}
and	the uniqueness of the solution	to	\eqref{I-invariant----}	follows by using arguments similar to that of	\cite[Lemma	1]{YV22}.
\end{remark}

\subsection{Proof	of	Theorem	\ref{TH1perturbation}}
\subsubsection*{Penalization}

Let $\epsilon>0$	and	$u_0	\geq	\psi$. Consider the following  approximating problem:
\begin{align}\label{1}
%\left\{ \begin{array}{l}
u_\epsilon(t)+\displaystyle\int_0^t(A(u_\epsilon)+\mathbf{F}(u_\epsilon)-\dfrac{1}{\epsilon}[(u_\epsilon-\psi)^-]^{p-1}-f)ds=u_0+\displaystyle\int_0^t\widetilde{G}(u_\epsilon)dW(s) %\\[1.2ex]
%u_\epsilon(0)=u_0,
%\end{array}
%%\right.
\end{align}
where  $\widetilde{G}(u_\epsilon)=G(\max(u_\epsilon,\psi))$. 
Note that $\widetilde{G}$ satisfies also  H$_{1}$ and H$_{3}$ and H$_{4}$ since $\widetilde{G}(\psi)=G(\psi)=0$. Indeed, 
For any $u,v \in V$,
\begin{align*}
&\Vert \widetilde{G}(u)-\widetilde{G}(v)\Vert_{L_2(H_0,H)}^2 \leq L_G \Vert \max(u,\psi)-\max(v,\psi)\Vert_H^2  \leq L_G \Vert u-v\Vert_H^2,\\
 &\Vert \widetilde{G}(u)\Vert_{L_2(H_0,H)}^2 \leq M+2L_G\|\psi\|_{H}^2+2L_G\Vert u\Vert_H^2 = \widehat{M}+L_G\Vert u\Vert_H^2,
\end{align*} where $\widehat{M}>0$ is a $L^\infty(\Omega_T)$-predictable element, depending only on the data.
\\[0.2cm]
Consider $\bar{A}(u_\epsilon)=A(u_\epsilon)+\mathbf{F}(u_\epsilon)-\dfrac{1}{\epsilon}[(u_\epsilon-\psi)^-]^{p-1}-f$ and note that:
\begin{itemize}
	\item By construction, $\bar{A}$ is an operator defined on $V$ with values in $V^\prime$. 
	\item  Since $x \mapsto -x^-$ is non-decreasing, $(\lambda_T+L_F) Id+\bar{A}$ is T-monotone. 
	\item The structure of the penalization operator yields that $\bar{A}$	satisfies  H$_{2,4}$.
	\item  (Coercivity): Note that for any $\delta>0$, there exists   $C_{\delta,\epsilon} >0$ such that: $\forall v \in V,$ 
	\begin{align*}
	\langle f,v\rangle \leq & C_\delta \Vert f\Vert_{V^\prime}^{p^\prime}+ \delta\Vert v\Vert_{V}^p \quad,
	\\
	\langle-\dfrac{1}{\epsilon}[(v-\psi)^-]^{p-1},v\rangle \geq &  \langle-\dfrac{1}{\epsilon}[(v-\psi)^-]^{p-1},\psi\rangle
	\geq  -\delta C \Vert v\Vert _{V}^{p}- C_{\delta,\epsilon} \Vert \psi\Vert _{L^{p}(D)}^{p},
	\end{align*}
	where $C$ is related  to the continuous embedding of $V$ in $L^{p}(D)$.
	Denote by $\tilde{l}_1=l_1+C_\delta \Vert f\Vert_{V^\prime}^{p^\prime}+C_{\delta,\epsilon} \Vert \psi\Vert _{L^{p}(D)}^{p}+\dfrac{CK}{2}$. It is a $L^\infty(\Omega_T)$ predictable element thanks to the assumptions on $f$ and $\psi$. Therefore, by a convenient choice of $\delta$, $\bar{A}$ satisfies $H_{2,1}$ by considering $\tilde{l}_1$ instead  of $l_1$	and	$\lambda-L_F-\dfrac{1}{2}\in	\mathbb{R}$	instead	of	$\lambda.$
	\item $\forall v \in V,$
$	\| -\dfrac{1}{\epsilon}[(v-\psi)^-]^{p-1} \|_{L^{p^\prime}(D)} =\dfrac{1}{\epsilon}\|(v-\psi)^- \|_{L^{p}(D)}^{p-1} \leq C_\epsilon\left(\|v\|_{L^{p}(D)}^{p-1}+\|\psi\|_{L^{p}(D)}^{p-1} \right)$ %%\\
%%	&\leq C_\epsilon\left(\|v\|_{L^{\tilde q}(D)}^{p-1}+\|\psi\|_{L^{\tilde q}(D)}^{p-1}\right) +C_p
	%%%%since $\tilde q <p$ may be possible.
	 Now, since the embeddings of $L^{p^\prime}(D)$ in $V^\prime$ and of  $V$  in $L^{p}(D)$ are continuous,  
	\begin{align*}
	\| -\dfrac{1}{\epsilon}[(v-\psi)^-]^{p-1} \|_{V^\prime}  \leq C_\epsilon\left(\|v\|_{V}^{p-1}+\|\psi\|_{V}^{p-1}\right)
	\end{align*} 
		and Assumption H$_{2,3}$ is satisfied with $\bar{K}$ replaced by $\bar{K}+C_\epsilon$ and $\bar{K}$ by $\widetilde{g}=\bar{K}+C_\epsilon \Vert \psi \Vert_{V}^{p-1}+\Vert	f\Vert_{V^\prime}$ which is a predictable element of $L^{\infty}(\Omega_T)$. Moreover,
$
\Vert	\mathbf{F}(v)\Vert_{V^\prime}\leq	C\Vert	\mathbf{F}(v)\Vert_H\leq	L_F\Vert	v\Vert_H+CK.$	Therefore
	$$	\Vert	\bar{A}(v)\Vert_{V^\prime}^{p^\prime}\leq	L_F\Vert	v\Vert_H^{p^\prime}+\tilde{g}+CK+C_\epsilon\|v\|_{V}^{p}.	$$
%%	\item  (The noise): Let us  denote by $U=Q^{\frac{1}{2}}(H)$, we recall  that $U$ is a  separable Hilbert space  endowed with the scalar product $ (u,v)_U=(Q^{-\frac{1}{2}}u,Q^{-\frac{1}{2}}u)_H$ (see \cite[Prop. C.0.3 p. 221]{RoK15}) and  note that $\widetilde{G} \in HS(Q^{\frac{1}{2}}(H),H)$. Since $(W(t))_{t\in [0,T]}$ is a Wiener process in $H$ with  a nuclear covariance operator $Q$ then $(W(t))_{t\in [0,T]}$ is a  Cylindrical Wiener  process with a  covariance operator $I$ in $U$.
\end{itemize}

By \cite[Th. 5.1.3 p.125 ]{RoK15}	with	$\beta=p^\prime$, we	get
\begin{lemma}\label{lemma-exis-u-epsilo,2}
	For all $\epsilon>0$, there exists  a unique solution $u_\epsilon \in L^p(\Omega_T;V)$ predictable such that $u_\epsilon \in L^2(\Omega;C([0,T],H))$  and satisfying \eqref{1} P-a.s. in $\Omega$ for all $t\in[0,T]$. Moreover,	$(u_\epsilon(t))_{t\in[0,T]}$	(solution	to	\eqref{1})	is		a	Markov	process.
\end{lemma}

%%%\\
%Moreover, thanks to  \cite[ Th. 4.2.5 p.91]{Liu-Rock}, $(u_\epsilon)_{\epsilon>0}$ is %%bounded in  $ L^p(\Omega_T,V)\cap L^2(\Omega_T,H)$. \\
%%Thanks to Assumptions  $H_{2,3}$, we get the following lemma.
\begin{lemma}\label{lem1}
	 $(u_\epsilon)_{\epsilon>0}$	and	
	 $(A(u_\epsilon))_{\epsilon>0}$ are bounded in  $ L^p(\Omega_T;V)\cap C([0,T],L^2(\Omega,H))$	and $L^{p^\prime}(\Omega_T;V^\prime)$,	respectively. 
	
\end{lemma}
\begin{proof}
Let $\epsilon >0$ and	note that 
	\begin{align*}
	&u_\epsilon(t)-\psi+\displaystyle\int_0^t\Big(A(u_\epsilon)+\mathbf{F}(u_\epsilon)-\dfrac{1}{\epsilon}[(u_\epsilon-\psi)^-]^{p-1})ds=
	u_0-\psi +\int_0^t  fds+\displaystyle\int_0^t\widetilde{G}(u_\epsilon)dW(s). 
	\end{align*}
	It\^o's stochastic energy (see \textit{e.g.} \cite[Thm. 2.3]{Pardoux_Ito}) yields
	\begin{align*}
	&\|u_\epsilon-\psi\|_{H}^2(t) - \|u_0-\psi\|^2_{H}+2\displaystyle\int_0^t \langle A(u_\epsilon)+\mathbf{F}(u_\epsilon),u_\epsilon-\psi\rangle ds 
	+\dfrac{2}{\epsilon}\int_0^t \!\!\!\int_D[(u_\epsilon-\psi)^-]^{p}dxds
	\\& =
	2\int_0^t \langle f, u_\epsilon-\psi \rangle ds+2\displaystyle\int_0^t  \Big(u_\epsilon-\psi,\widetilde{G}(u_\epsilon)dW(s)\Big)_H + \int_0^t \|\widetilde{G}(u_\epsilon)\|_{L_2(H_0,H)}^2 ds. 
	\end{align*}
	Note that $	\langle A(u_\epsilon),u_\epsilon-\psi\rangle  \geq \alpha \Vert u_\epsilon\Vert_V^p - \lambda \Vert u_\epsilon\Vert_H^2-l_1 -  \langle A(u_\epsilon),\psi\rangle.
	$	Thus
	\begin{align}\label{estimate-A-coerc}
\langle A(u_\epsilon),u_\epsilon-\psi\rangle	\geq& \alpha \Vert u_\epsilon\Vert_V^p - \lambda \Vert u_\epsilon\Vert_H^2-l_1  - \bar K\Vert u_\epsilon\Vert_V^{p-1}\Vert \psi\Vert_V-\bar K\Vert \psi\Vert_V
	\notag\\
	\geq& \frac{\alpha}{2} \Vert u_\epsilon\Vert_V^p - \lambda \Vert u_\epsilon\Vert_H^2-l_1   - \bar K\Vert \psi\Vert_V-C(\bar{K},\alpha)\Vert \psi\Vert_V^p.
	\end{align}
	Additionally,	$\langle \mathbf{F}(u_\epsilon),u_\epsilon-\psi\rangle  \geq 	-(L_F+1)\Vert	u_\epsilon\Vert_H^2-C\mathbf{F}(0)^2-\langle \mathbf{F}(u_\epsilon),\psi\rangle.
	$
	By	using	\eqref{H-F},	we	get	$	\langle \mathbf{F}(u_\epsilon),u_\epsilon-\psi\rangle  \geq		-(L_F+1)\Vert	u_\epsilon\Vert_H^2-C\mathbf{F}(0)^2-L_F^2\Vert	u_\eps\Vert_H^2-\Vert	\psi\Vert_H^2,
		$	hence
	\begin{align}\label{estimate-F}
	\langle \mathbf{F}(u_\epsilon),u_\epsilon-\psi\rangle 	\geq&-(2L_F^2+C)\Vert	u_\epsilon\Vert_H^2-[C\mathbf{F}(0)^2+\Vert	\psi\Vert_H^2].
	\end{align}
	 Thus, for any positive $\gamma$, Young's inequality yields the existence of a positive constant $C_\gamma$ such that 
	\begin{align}\label{estimate-key-pen}
	&\E\|u_\epsilon-\psi\|_{H}^2(t) +2\displaystyle \E\int_0^t \frac{\alpha}{2} \Vert u_\epsilon\Vert_V^p (s) ds	+\dfrac{2}{\epsilon}\E\int_0^t \!\!\!\int_D[(u_\epsilon-\psi)^-]^{p}dxds\notag
	\\
	&\leq  \E\|u_0-\psi\|^2_{H}+ 
	C(\lambda+L_F^2+1) \E\int_0^t \Vert u_\epsilon\Vert_H^2 (s) ds\notag\\&\quad +2t[l_1   + \bar K\Vert \psi\Vert_V+C(\bar{K},\alpha)\Vert \psi\Vert_V^p +C\mathbf{F}(0)^2+\Vert	\psi\Vert_H^2+C_\gamma\Vert		f\Vert_{V^\prime}^{p^\prime}] 
\notag	\\
	&\quad+ \gamma \E\int_0^t\|u_\epsilon-\psi \|^p_V(s) ds+\E\int_0^t \|\widetilde{G}(u_\epsilon)-\widetilde{G}(\psi)\|_{L_2(H_0,H)}^2 ds. 
\notag	\\ &\leq  
	C \E\int_0^t \Vert u_\epsilon - \psi\Vert_H^2 (s) ds+ \frac{\alpha}{2}\E\int_0^t\|u_\epsilon\|^p_V (s)ds+ \mathbf{C}t,
	\end{align}
	where	$\mathbf{C}=2[l_1   + \bar K\Vert \psi\Vert_V+C(\bar{K},\alpha)\Vert \psi\Vert_V^p +C\mathbf{F}(0)^2+C\Vert	\psi\Vert_H^2+C_\gamma\Vert		f\Vert_{V^\prime}^{p^\prime}]$	and	$C>0$	is	a	generic	constant,	after		
	 a suitable choice of $\gamma$. 
	Then, the first part of the lemma is proved by  Gronwall's lemma, and the second one by adding H$_{2,3}$ to the first estimate.
\end{proof}
Notice	that	\eqref{estimate-key-pen} does not guarantee the necessary bound on the penalized term, we will therefore prove the following lemma.
\begin{lemma}\label{lem2} Assume	that	$\widetilde{h}^-\in L^{p^\prime}(D)$,	then
	$(\dfrac{1}{\epsilon}[(u_\epsilon-\psi)^-]^{p-1})_{\epsilon>0}$ is bounded in $L^{p^\prime}(\Omega_T\times D)$	and		$(u_\epsilon)_{\epsilon>0}$ is  a Cauchy sequence in the space $L^2(\Omega;C([0,T],H))$.
\end{lemma}
\begin{proof}
	Let $\delta>0$ and consider the following approximation  
	\begin{eqnarray}\label{2ch3}
	F_\delta(r)= \left\{ \begin{array}{l}
	r^2-\dfrac{\delta^2}{6} \quad\quad\quad   if \hspace*{0.2cm} r\leq -\delta, \\
	-\dfrac{r^4}{2\delta^2}-\dfrac{4r^3}{3\delta}\quad if \hspace*{0.2cm}  -\delta\leq r \leq 0, \\
	0 \quad\quad\quad\quad \quad\quad if \hspace*{0.2cm}  r\geq 0.
	\end{array}
	\right.
	\end{eqnarray} 
	  		Set $\varphi_\delta(v)=\int_DF_\delta(v(x))dx; v\in H$ and  denote by $S$ the set $ \{ u_\epsilon\leq \psi \}$. Applying  Ito's formula  to the process $u_\epsilon-\psi$, one gets for any $t\in[0,T]$
		\begin{align*}
	&\varphi_\delta(u_\epsilon(t)-\psi(t))+\displaystyle\int_0^t \langle A(u_\epsilon)-A(\psi),F_\delta^\prime(u_\epsilon-\psi)\rangle ds+\displaystyle\int_0^t \langle \mathbf{F}(u_\epsilon)-\mathbf{F}(\psi),F_\delta^\prime(u_\epsilon-\psi)\rangle ds\\&-\dfrac{1}{\epsilon}\displaystyle\int_0^t\langle [(u_\epsilon-\psi)^-]^{p-1},F_\delta^\prime(u_\epsilon-\psi)\rangle ds=\overbrace{\varphi_\delta(u_\epsilon(0)-\psi(0))}^{=0}+\overbrace{\displaystyle\int_0^t\langle \widetilde{h}^+,F_\delta^\prime(u_\epsilon-\psi)\rangle ds}^{\leq 0}
\\&	-\displaystyle\int_0^t\langle \widetilde{h}^-,F_\delta^\prime(u_\epsilon-\psi)\rangle ds+\displaystyle\int_0^t (\widetilde{G}(u_\epsilon)dW_s ,F_\delta^\prime(u_\epsilon-\psi)) +\dfrac{1}{2}\displaystyle\int_0^t  Tr(F_\delta^{\prime\prime}(u_\epsilon-\psi)\widetilde{G}(u_\epsilon)Q\widetilde{G}(u_\epsilon)^*) ds.
	\end{align*}
	Since $ \widetilde{G}(u_\epsilon)= \widetilde{G}(\psi)=0$ on the set $S$, we deduce  $\dfrac{1}{2}\displaystyle\int_0^t  Tr(F_\delta^{\prime\prime}(u_\epsilon-\psi)\widetilde{G}(u_\epsilon)Q\widetilde{G}(u_\epsilon)^*) ds=0.$ 	Taking the expectation, one has 
	\begin{align*}
	\E\varphi_\delta(u_\epsilon(t)-\psi(t))&+\E\displaystyle\int_0^t \langle A(u_\epsilon)-A(\psi),F_\delta^\prime(u_\epsilon-\psi)\rangle ds+\E\displaystyle\int_0^t \langle \mathbf{F}(u_\epsilon)-\mathbf{F}(\psi),F_\delta^\prime(u_\epsilon-\psi)\rangle ds\\&-\dfrac{1}{\epsilon}\E\displaystyle\int_0^t\langle [(u_\epsilon-\psi)^-]^{p-1},F_\delta^\prime(u_\epsilon-\psi)\rangle ds\leq  \E \displaystyle\int_0^t\langle- \widetilde{h}^-,F_\delta^\prime(u_\epsilon-\psi)\rangle ds.
	\end{align*}
	
	Recall	that   $F^\prime_\delta(u_\epsilon-\psi)$ converges to $-2(u_\epsilon-\psi)^-$ in $V$	a.e. $t\in [0,T]$ and P-a.s. 	(see	\cite[Proof	of	Lemma	3]{YV22}).
Thus,	dominated convergence theorem  ensures that for any $t\in[0,T]$
		\begin{align}
		\E\Vert(u_\epsilon-\psi)^-(t)\Vert_{L^2(D)}^2&+\dfrac{2}{\epsilon} \E\displaystyle\int_0^t\Vert (u_\epsilon-\psi)^-(s)\Vert_{L^{p}(D)}^{p}ds \label{Pen} \\  &\leq 2\E\displaystyle\int_0^t\langle \widetilde{h}^-(s), (u_\epsilon-\psi)^-(s)\rangle ds+2(\lambda_T+L_F) \E\int_0^t\Vert (u_\epsilon-\psi)^-(s)\Vert_H^2ds.\nonumber
		\end{align}
	Recall	that $1<p<2$,  from Gr\"onwall's lemma applied to \eqref{Pen}, one gets
	\begin{align*}
			\dfrac{1}{\epsilon}\Vert  (u_\epsilon-\psi)^-\Vert_{L^{p}(\Omega_T\times D)}^p &= \dfrac{1}{\epsilon} \E\displaystyle\int_0^T \Vert (u_\epsilon-\psi)^-(s)\Vert_{L^{p}(D)}^{p}ds\leq  C_T\Vert \widetilde{h}^-\Vert_{L^{p^\prime}(\Omega_T\times D)}\Vert  (u_\epsilon-\psi)^-\Vert_{L^{p}(\Omega_T\times D)},
			\end{align*}
		hence $\dfrac{1}{\epsilon}\Vert  (u_\epsilon-\psi)^-\Vert_{L^{p}(\Omega_T\times D)}^{p-1} \leq C_T\Vert h^-\Vert_{L^{p^\prime}(\Omega_T\times D)}.$
		On the other hand, we have  $$\Vert  \dfrac{1}{\epsilon}[(u_\epsilon-\psi)^-]^{p-1} \Vert_{L^{p^\prime}(\Omega_T\times D)}=\dfrac{1}{\epsilon}( \E\displaystyle\int_0^T\int_D[(u_\epsilon-\psi)^-]^{(p-1)p^\prime}dxds)^{\frac{1}{p^\prime}}= \dfrac{1}{\epsilon}\Vert  (u_\epsilon-\psi)^-\Vert_{L^{p}(\Omega_T\times D)}^{p-1}.$$ 
		Consequently, $(\dfrac{1}{\epsilon}[(u_\epsilon-\psi)^-]^{p-1})_{\epsilon>0}$ is bounded in $L^{p^\prime}(\Omega_T\times D)$.\\
		
$\bullet$	Let $1>\epsilon\geq \delta >0$ and consider the process $u_\epsilon-u_\delta$, which satisfies the following 
	\begin{align*}
	u_\epsilon(t)-u_\delta(t)&+\displaystyle\int_0^t [(A(u_\epsilon)-A(u_\delta))+ (\mathbf{F}(u_\epsilon)-\mathbf{F}(u_\delta))
	\\&+ (-\dfrac{1}{\epsilon}[(u_\epsilon-\psi)^-]^{p-1}+\dfrac{1}{\delta}[(u_\delta-\psi)^-]^{p-1})]ds=\displaystyle\int_0^t (\widetilde{G}(u_\epsilon)-\widetilde{G}(u_\delta))dW(s). 
	\end{align*}
	Applying Ito's formula with $F(t,v)=\dfrac{1}{2}\Vert v\Vert_H^2$, one gets for any $t\in [0,T]$
	\begin{align*}
	\dfrac{1}{2}\Vert (u_\epsilon-u_\delta)(t)\Vert_H^2&+\displaystyle\int_0^t \langle A(u_\epsilon)-A(u_\delta),u_\epsilon-u_\delta\rangle ds+\int_0^t(\mathbf{F}(u_\epsilon)-\mathbf{F}(u_\delta),u_\epsilon-u_\delta)ds\\
	&+\displaystyle\int_0^t\langle  -\dfrac{1}{\epsilon}[(u_\epsilon-\psi)^-]^{p-1}+\dfrac{1}{\delta}[(u_\delta-\psi)^-]^{p-1}, u_\epsilon-u_\delta \rangle ds\\
	&= \displaystyle\int_0^t\langle (\widetilde{G}(u_\epsilon)-\widetilde{G}(u_\delta))dW(s), u_\epsilon-u_\delta\rangle +\dfrac{1}{2}\displaystyle\int_0^t\Vert \widetilde{G}(u_\epsilon)-\widetilde{G}(u_\delta)\Vert_{L_2(H_0,H)}^2ds. 
	\end{align*}
	%% \dfrac{1}{2}\displaystyle\int_0^t\vert\widetilde{G}(u_\epsilon,\cdot)-\widetilde{G}(u_\delta,\cdot)\vert_Q^2ds=\dfrac{1}{2}\displaystyle\int_0^t  Tr(\{\widetilde{G}(u_\epsilon,\cdot)-\widetilde{G}(u_\delta,\cdot)\}Q\{\widetilde{G}(u_\epsilon,\cdot)-\widetilde{G}(u_\delta,\cdot)\}^*) ds
Recall	that	
$	\displaystyle\int_0^t\vert(\mathbf{F}(u_\epsilon)-\mathbf{F}(u_\delta),u_\epsilon-u_\delta)\vert	ds\leq	L_F	\int_0^t\Vert (u_\epsilon-u_\delta)(s)\Vert_H^2 ds
.$	We argue as in the proof of \cite[Lemma	4]{YV22} to	estimate	the	other	terms.	Thus, we deduce 
	$$ \E\displaystyle\sup_{t\in[0,T]}\Vert (u_\epsilon-u_\delta)(t)\Vert_H^2 \leq C\epsilon^{\frac{1}{p-1}}+C\displaystyle\int_0^T\E\sup_{\tau\in[0,s]}\Vert (u_\epsilon-u_\delta)(\tau)\Vert_H^2 ds$$
	and  Gr\"onwall's  lemma ensures that $(u_\epsilon)_{\epsilon>0}$ is a Cauchy sequence in the space $L^2(\Omega;C([0,T],H))$. 
\end{proof}
From Lemma \ref{lem1}	and	 Lemma \ref{lem2}, we deduce the following result.
\begin{lemma}\label{Lem2.4}
	There exist $u \in L^p(\Omega_T;V)\cap L^2(\Omega; C([0,T],H)) \cap  \mathcal{N}^2_W(0,T,H)$\footnote{$\mathcal{N}^2_W(0,T,H)$ denotes the space of all predictable process of $L^2(\Omega_T;H)$ (see \cite[p. 36]{RoK15}).} and $ (\rho,\chi ) \in L^{p^\prime}(\Omega_T;L^{p^\prime}(D))\times L^{p^\prime}(\Omega_T;V^\prime)$, each one  predictable, such that the following convergences hold	(as $\epsilon\rightarrow 0$), up to sub-sequences denoted by the same way,
	\begin{align}
	u_\epsilon &\rightharpoonup u \quad \text{in} \quad  L^p(\Omega_T;V),\label{3-8}\\
	u_\epsilon &\rightarrow u  \quad \text{in} \quad L^2(\Omega; C([0,T],H)),\label{4}\\ 
	A(u_\epsilon) &\rightharpoonup \chi  \quad \text{in} \quad L^{p^\prime}(\Omega_T;V^\prime),\label{5}\\
	-\dfrac{1}{\epsilon}[(u_\epsilon-\psi)^-]^{p-1} &\rightharpoonup \rho,\quad \rho \leq 0 \quad \text{in} \quad L^{p^\prime}(\Omega_T\times D).\label{7}
	\end{align}
Moreover,	$u(0)=u_0$	and	$u\in	K$.		Finally,	 we get
\begin{align}\label{cv8}
	\displaystyle\int_0^\cdot \widetilde{G}(u_\epsilon) dW(s)\rightarrow  \displaystyle\int_0^\cdot G(u) dW(s)  \text{	and	}		\mathbf{F}(u_\epsilon)	\to	\mathbf{F}(u)	\text{	in	}	L^2(\Omega; C([0,T],H)).		
\end{align}
	\end{lemma}

\begin{proof}
	By  compactness  with respect to the weak topology in the spaces $ L^p(\Omega_T;V)$, $L^{p^\prime}(\Omega_T;V^\prime)$ and $L^{p^\prime}(\Omega_T\times D)$, there exist $u \in L^p(\Omega_T;V)$, $\chi \in L^{p^\prime}(\Omega_T;V^\prime)$ and $\rho \in L^{p^\prime}(\Omega_T\times D) $ such that $\eqref{3-8}$, $\eqref{5}$ and \eqref{7}  hold (for  sub-sequences). 
	Thanks to  Lemma \ref{lem2}, we get   the strong convergence of $u_\epsilon$ to $u$ in $ L^2(\Omega; C([0,T],H)) \hookrightarrow L^2(\Omega_T\times D)$.
	Moreover, since $u_\epsilon \in \mathcal{N}_W^2(0,T,H)$ 	and	$(A(u_\epsilon))_\epsilon$ is predictable with values in $V^\prime$,	the same applies to  	 $ u \in \mathcal{N}_W^2(0,T,H)$ and	$\chi$. Furthermore,	 $-\dfrac{1}{\epsilon}[(u_\epsilon-\psi)^-]^{p-1}$  is a  predictable process with values in $L^{p^\prime}( D)$. Hence $\rho$ is a predictable process with values in $L^{p^\prime}( D) $ and $\rho \leq 0$ since the set of non positive functions of $L^{p^\prime}(\Omega_T\times D)$ is a closed convex subset of $L^{p^\prime}(\Omega_T\times D)$.	In	addition,	 since  $u_\epsilon$ converges to $u$ in $ L^2(\Omega; C([0,T],H))$ then $u_\epsilon(0)=u_0$ converges to $ u(0)$ in $ L^2(\Omega;H)$ and $u(0)=u_0$ in $L^2(\Omega;H)$.
	Thanks   to Lemma \ref{lem2}, we deduce that  $(u_\epsilon-\psi)^-\rightarrow (u-\psi)^-=0$ in $L^{p}(\Omega_T\times D)$ and $u\in K$.	Finally,	note	that
		$
	\E\displaystyle\sup_{t\in[0,T]}\Vert \displaystyle\int_0^t ( \mathbf{F}(u_\epsilon)-\mathbf{F}(u))ds\Vert_H^2\leq  L_F^2 \E\displaystyle\int_0^T\Vert u_\epsilon-u\Vert_H^2ds$.
		By Burkh\"older-Davis-Gundy's inequality 	and	$H_3$, one gets
	\begin{align*}
	\E\displaystyle\sup_{t\in[0,T]}\Vert \displaystyle\int_0^t ( \widetilde{G}(u_\epsilon)-\widetilde{G}(u)) dW(s)\Vert_H^2&\leq 3\E \displaystyle\int_0^T \Vert \widetilde{G}(u_\epsilon)-\widetilde{G}(u) \Vert_{L_2(H_0,H)}^2 ds\\& \leq 3 L_G \E\displaystyle\int_0^T\Vert u_\epsilon-u\Vert_H^2ds.
	\end{align*}

	Since $u_\epsilon \rightarrow u$ in $ L^2(\Omega, C([0,T],H))$ and $u \in K$, one deduces	\eqref{cv8}.
\end{proof}

\begin{lemma}\label{lem7}
$\rho(u-\psi)=0$  a.e. in $\Omega_T\times	D$ and, $\forall v\in K$, $\rho(u-v) \geq 0$ a.e. in $\Omega_T\times	D$.
\end{lemma}
\begin{proof}
	On one hand, by Lemma \ref{lem2}, we have 
	$$  0\leq -\dfrac{1}{\epsilon} \E\displaystyle\int_0^t \langle [(u_\epsilon-\psi)^-]^{p-1}, u_\epsilon-\psi\rangle ds=\dfrac{1}{\epsilon} \E\displaystyle\int_0^t \Vert (u_\epsilon-\psi)^-(s)\Vert_{L^{p}}^{p}ds \leq C\epsilon^{p^\prime-1}\rightarrow 0.$$
	On the other hand, by Lemma \ref{Lem2.4} $-\dfrac{1}{\epsilon}[(u_\epsilon-\psi)^-]^{p-1} \rightharpoonup  \rho$ in $L^{p^\prime}(\Omega_T\times D)$ and $ u_\epsilon-\psi \rightarrow u-\psi$ in $L^p(\Omega_T\times D)$.  Hence $\E \int_0^T \int_D\rho(u-\psi)dxdt=0$ and $\rho(u-\psi)=0$ a.e.	since the integrand is always non-positive.	
		One finishes  the proof by noticing that if $v\in K$,  one has a.e. in $\Omega_T$ that, 
	$\langle \rho,u-v\rangle  = \underbrace{\langle \rho,u-\psi\rangle }_{=0}+\underbrace{\langle \rho,\psi-v\rangle }_{\geq 0}\geq 0.$
\end{proof}
  We have for any $t\in [0,T]$,	
\begin{align*}
u_\epsilon(t)+\displaystyle\int_0^t(A(u_\epsilon)-\mathbf{F}(u_\epsilon) -\dfrac{1}{\epsilon}[(u_\epsilon-\psi)^-]^{p-1}- f)ds&=u_0+\displaystyle\int_0^t \widetilde{G}(u_\epsilon) dW(s), 
\end{align*}
and $  u(t)+\displaystyle\int_0^t(\chi-\mathbf{F}(u)+\rho-f)ds=u_0+\displaystyle\int_0^t \widetilde{G}(u) dW(s),$
after	passing	to	the	limit	as	$\epsilon\downarrow	0$	in	\eqref{1}.
Our aim now is to prove that $A(u)=\chi$.
Note	that
$$  u_\epsilon(t)-u(t)+\displaystyle\int_0^t((A(u_\epsilon) -\chi)-(\mathbf{F}(u_\epsilon)-\mathbf{F}(u))+(-\dfrac{1}{\epsilon}[(u_\epsilon-\psi)^-]^{p-1}-\rho))ds=\displaystyle\int_0^t(\widetilde{G}(u_\epsilon)-\widetilde{G}(u))dW(s).$$
Note that  $ (A(u_\epsilon)-\chi)+(-\dfrac{1}{\epsilon}[(u_\epsilon-\psi)^-]^{p-1}-\rho) \in L^{p^\prime}(\Omega_T;V^\prime)$	and	$\mathbf{F}(u_\epsilon)-\mathbf{F}(u) \in	L^2(\Omega_T;H)$	are predictables and that 
$ \displaystyle\int_0^t(\widetilde{G}(u_\epsilon)-\widetilde{G}(u)) dW(s)$ is a square integrable $\mathcal{F}_t-$martingale. Thus, we can apply It\^o's formula (\textit{e.g.} \cite[Thm. 2.3]{Pardoux_Ito})  to the process $u_\epsilon-u$ with $F(v)=\dfrac{1}{2} \Vert v\Vert_H^2$ to get
\begin{align*}
\dfrac{1}{2}\Vert &(u_\epsilon-u)(t)\Vert_H^2+\overbrace{\int_0^t(\mathbf{F}(u_\epsilon)-\mathbf{F}(u),u_\epsilon-u)ds}^{I_0}+\overbrace{\displaystyle\int_0^t \langle A(u_\epsilon) -\chi, u_\epsilon-u\rangle ds}^{I_1}\\
& +\overbrace{\displaystyle\int_0^t \langle -\dfrac{1}{\epsilon}[(u_\epsilon-\psi)^-]^{p-1}-\rho, u_\epsilon-u\rangle ds}^{I_2}\\
&=\overbrace{\displaystyle\int_0^t \langle (\widetilde{G}(u_\epsilon)-\widetilde{G}(u)) dW(s), u_\epsilon-u\rangle }^{I_3} +\overbrace{\dfrac{1}{2}\displaystyle\int_0^t \Vert\widetilde{G}(u_\epsilon)-\widetilde{G}(u)\Vert_{L_2(H_0,H)}^2 ds}^{I_4}.
\end{align*}
Let us consider in the sequel a given $v\in L^p(\Omega_T;V)\cap L^2(\Omega;C[0,T],H))$ and $t\in ]0,T]$.
\begin{itemize}
	\item  Note that $I_1=\displaystyle\int_0^t \langle A(u_\epsilon), u_\epsilon\rangle ds-\displaystyle\int_0^t \langle A(u_\epsilon), u\rangle ds-\displaystyle\int_0^t \langle \chi, u_\epsilon-u\rangle ds$ and  
	\begin{align*}
	\displaystyle\int_0^t \langle A(u_\epsilon), u_\epsilon\rangle ds\geq \displaystyle\int_0^t \langle A(v), u_\epsilon-v\rangle ds +\displaystyle\int_0^t \langle A(u_\epsilon),v\rangle ds-\lambda_T\int_0^t\Vert u_\epsilon -v \Vert_H^2ds.
	\end{align*}
	
	\item  $\E(I_2)= \E\displaystyle\int_0^t \langle -\dfrac{1}{\epsilon}[(u_\epsilon-\psi)^-]^{p-1}, u_\epsilon-u\rangle ds-\E\displaystyle\int_0^t \langle \rho, u_\epsilon-u\rangle ds \\ \hspace*{1cm}\geq  \E\displaystyle\int_0^t \langle -\dfrac{1}{\epsilon}[(u_\epsilon-\psi)^-]^{p-1}, \psi-u\rangle ds-\E\displaystyle\int_0^t \langle \rho, u_\epsilon-u\rangle ds.$
\end{itemize}
 Since $I_3$  is a $\mathcal{F}_t$-square	integrable	martingale then $\E(I_3)=0$
and	thanks to $H_3$ we have $\E(I_4)\leq L_G\E\displaystyle\int_0^t\Vert u_\epsilon(s)-u(s)\Vert_H^2ds .$	Furthermore,		\eqref{H-F}	gives	$\vert	I_0\vert	\leq	L_F\displaystyle\int_0^t\Vert u_\epsilon -u \Vert_H^2ds.$
By gathering the previous computation and taking the expectation, one has for any $t\in]0,T]$
\begin{align*}
&\dfrac{1}{2}\E\Vert (u_\epsilon-u)(t)\Vert_H^2+ \E\displaystyle\int_0^t \langle A(v), u_\epsilon-v\rangle ds +\E\displaystyle\int_0^t \langle A(u_\epsilon),v-u\rangle ds-\E\displaystyle\int_0^t \langle \chi, u_\epsilon-u\rangle ds \\&+\E\displaystyle\int_0^t \langle -\dfrac{1}{\epsilon}[(u_\epsilon-\psi)^-]^{p-1}, \psi-u\rangle ds-\E\displaystyle\int_0^t \langle \rho, u_\epsilon-u\rangle ds\\
&\quad \leq (L_G+L_F)\E\displaystyle\int_0^t\Vert u_\epsilon(s)-u(s)\Vert_H^2ds+\lambda_T\E\int_0^t\Vert u_\epsilon -v \Vert_H^2ds.
\end{align*}
By  passing to the  limit as $\epsilon \rightarrow 0$,  thanks to Lemmas \ref{Lem2.4} and \ref{lem7},  we get  
$$ \E\displaystyle\int_0^T\langle A(v)-\chi, u-v\rangle ds  \leq \lambda_T^+\E\int_0^T\Vert u -v \Vert_H^2ds+ \overbrace{\E\displaystyle\int_0^T\langle \rho, u-\psi\rangle ds}^{=0}. $$
We are now in a position  to use "Minty's trick" 	\textit{e.g.}	\cite[Lemma 2.13 p.35]{Roubicek}.		Indeed,		set	$v=u+\delta\varphi,	\varphi	\in	 L^p(\Omega_T;V)\cap L^2(\Omega;C([0,T],H)),	\delta>0$	to  conclude $$\E\displaystyle\int_0^T \langle A(u+\delta\varphi)-\chi, \varphi\rangle ds \geq \lambda_T^+\delta\E\int_0^T\Vert\varphi \Vert_H^2ds	\to_{\delta\to	0}	0.$$
	By	using	$H_{2,4}$,	we	get	$\E\displaystyle\int_0^T \langle A(u)-\chi, \varphi\rangle ds= 0,	\forall	\varphi\in	 L^p(\Omega_T;V)\cap L^2(\Omega;C([0,T],H))$,	which	yields	$\chi=A(u)$	and	ends	the	proof	of	Theorem	\ref{TH1perturbation}.

\begin{cor}\label{cor-tightness-bis} Let $K>0$, there	exists	$\mathcal{K}:=C(f,\psi,l_1,\bar{K},\alpha,p,\gamma,M,K)>0$	such	that
		\begin{align}\label{inequality-tight-Bis}
\dfrac{1}{\alpha}\E\Vert u(t)\Vert_H^2&+\E\int_0^t \Vert u(s)\Vert_V^pds    \leq	\dfrac{4}{\alpha}[	\|y_0\|^2_{H}+\mathcal{K}(t+1)	],
\end{align}
for		$\frac{L_G(1+K^2)}{2K^2}+\lambda+L_F\leq	0$	and	any	$t\in[0,T]$. Moreover,  there exists $S>0$ depending only on the data such that 
  \begin{align}
&\E	f_\epsilon(\Vert u(t)-\psi \Vert_H^2)+\dfrac{\alpha}{2}\displaystyle\E	\int_0^tf_\epsilon^\prime(\Vert	u(s)-\psi\Vert_H^2)\Vert u(s)\Vert_V^pds\notag\\&\leq f_\epsilon(\Vert u_0 -\psi\Vert_H^2)+ (2 \lambda+2L_F+\dfrac{L_G(1+K^2)}{K^2})\E	\int_0^t f_\epsilon^\prime(\Vert	u(s)-\psi\Vert_H^2)\Vert u(s)\Vert_H^2ds+St. \label{est-ergo-2}\end{align}  
If $G$ satisfies \eqref{bound-noise-p=2}, it holds
  \begin{align}
&\E	f_\epsilon(\Vert u(t)-\psi \Vert_H^2)+\dfrac{\alpha}{2}\displaystyle\E	\int_0^tf_\epsilon^\prime(\Vert	u(s)-\psi\Vert_H^2)\Vert u(s)\Vert_V^pds\notag\\&\leq f_\epsilon(\Vert u_0 -\psi\Vert_H^2)+ (2 \lambda+2L_F)\E	\int_0^t f_\epsilon^\prime(\Vert	u(s)-\psi\Vert_H^2)\Vert u(s)\Vert_H^2ds+(S+\mathbf{K})t.\label{est-ergo-3}\end{align}
\end{cor}
\begin{proof}
We	have
$u(t)-\psi+\displaystyle\int_0^t\Big(A(u)+\mathbf{F}(u)+\rho)ds
 =
u_0-\psi +\int_0^t fds+\displaystyle\int_0^tG(u)dW(s). 
$	It\^o's stochastic energy (\textit{e.g.} \cite[Thm. 2.3]{Pardoux_Ito}) yields
\begin{align}
&\|u-\psi\|_{H}^2(t)+2\displaystyle\int_0^t \langle A(u),u-\psi \rangle	ds +2\int_0^t \langle \mathbf{F}(u),u-\psi\rangle ds
+2\int_0^t \int_D\rho(u-\psi)dxds \notag
\\ =&\|u_0-\psi\|^2_{H}+
2\int_0^t \langle f, u-\psi \rangle ds+2\displaystyle\int_0^t  \Big(u-\psi,G(u)dW(s)\Big)_H + \int_0^t \|G(u)\|_{L_2(H_0,H)}^2 ds. \label{appendix-1}
\end{align}
\begin{itemize}
    \item[i.]     By	taking	the	expectation	and	using	Lemma	\ref{lem7},	we	obtain	for	any	$K>0$ 
\begin{align*}
&\dfrac{1}{2}\E\Vert u(t)-\psi \Vert_H^2+\E\displaystyle\int_0^t \langle A(u),u-\psi \rangle	ds+\E\int_0^t \langle \mathbf{F}(u),u-\psi\rangle ds\\&\leq	\dfrac{1}{2}\|u_0-\psi\|^2_{H}+
\E\int_0^t \langle f, u-\psi \rangle ds\\&\quad+\dfrac{M(1+K^2)}{2}t+	\dfrac{L_G(1+K^2)}{2K^2}\displaystyle\E\int_0^t\Vert	u(s)\Vert_H^2 ds. 
\end{align*}
Similarly	to	\eqref{estimate-A-coerc},	we get 
\begin{align}\label{coerv-A-psi}
\langle A(u),u-\psi\rangle  \geq\frac{\alpha}{2} \Vert u\Vert_V^p - \lambda \Vert u\Vert_H^2-l_1  - \bar K\Vert \psi\Vert_V-C(\bar{K},\alpha)\Vert \psi\Vert_V^p.    
\end{align}
By using $V\hookrightarrow H\hookrightarrow V^\prime$ and Young's inequality, we get \begin{align}
& \E\int_0^t \langle \mathbf{F}(u),u-\psi\rangle ds \leq L_F\E\int_0^t \Vert		u\Vert_H^2ds++ 
\E\int_0^t \langle	\mathbf{F}(u),\psi\rangle ds\notag\\
&\leq	L_F\E\int_0^t \Vert		u\Vert_H^2ds+L_F\E\int_0^t \Vert	u\Vert_{H} \Vert\psi\Vert_V ds+\E\int_0^t \Vert	\mathbf{F}(0)\Vert_H \Vert u\Vert_H ds+t \Vert	\mathbf{F}(0)\Vert_{H} \Vert\psi\Vert_V\notag\\
&\leq L_F\E\int_0^t \Vert		u\Vert_H^2ds+ 
(L_F\Vert\psi\Vert_V+\Vert	\mathbf{F}(0)\Vert_H)\E\int_0^t \Vert	u\Vert_{H}  ds+t \Vert	\mathbf{F}(0)\Vert_{H} \Vert\psi\Vert_V\notag\\
&\leq L_F\E\int_0^t \Vert		u\Vert_H^2ds+\frac{\alpha}{8}\E\int_0^t \Vert	u\Vert_{V}^p ds+ C(\alpha,p,\Vert	\mathbf{F}(0)\Vert_{H},\Vert\psi\Vert_V,L_F)t,\label{est-F-psi}
\end{align}
and		
\begin{align}\label{est-f-psi}
\E\displaystyle\int_0^t \vert\langle f, u-\psi \rangle\vert ds	\leq	\frac{\alpha}{8} \E\int_0^t\Vert u\Vert_V^pds+t[C(\alpha,p)\Vert	f\Vert_{V^\prime}^{p^\prime}+\Vert	f\Vert_{V^\prime}\Vert		\psi\Vert_V].    
\end{align}
Thus,	we	obtain	for	$K>0$,	after	using	the	above	estimates
\begin{align}
\E\Vert u(t)\Vert_H^2&+\alpha\E\int_0^t \Vert u(s)\Vert_V^pds    \notag\\&\leq	4\|u_0\|^2_{H}+\mathcal{K}(t+1)+4[	\frac{L_G(1+K^2)}{2K^2}+\lambda+L_F)]\displaystyle\E\int_0^t\Vert	u(s)\Vert_H^2 ds,\label{appenxix-2} 
\end{align}
where $\mathcal{K}>0$ depending only on the data.
If $\frac{L_G(1+K^2)}{2K^2}+\lambda+L_F\leq	0$,	one	has
\begin{align*}
\dfrac{1}{\alpha}\E\Vert u(t)\Vert_H^2&+\E\int_0^t \Vert u(s)\Vert_V^pds    \leq	\dfrac{1}{\alpha}[	4\|u_0\|^2_{H}+\mathcal{K}(t+1)	].
\end{align*}
\item[ii.] Let	us	consider		the	following	function		$f_\epsilon:x\mapsto	\dfrac{x}{1+\epsilon	x},	\epsilon>0	$ and recall	that	$f_\epsilon\in		C^2_b(\mathbb{R}_+)$.  By using \eqref{appendix-1}, Lemma	\ref{lem7} and applying  It\^o's	formula	to	the	process	$\Vert	u-\psi\Vert_H^2$, we obtain
	\begin{align}
&f_\epsilon(\Vert u(t)-\psi \Vert_H^2)-f_\epsilon(\Vert u_0 -\psi\Vert_H^2)+2\displaystyle\int_0^tf_\epsilon^\prime(\Vert	u(s)-\psi\Vert_H^2)\langle A(u), u-\psi\rangle ds\notag\\&+2\int_0^t f_\epsilon^\prime(\Vert	u(s)-\psi\Vert_H^2)\langle \mathbf{F}(u),u-\psi\rangle ds=2\displaystyle\int_0^tf_\epsilon^\prime(\Vert	u(s)-\psi\Vert_H^2)\langle f, u-\psi\rangle ds\notag\\
	&+2\displaystyle\int_0^tf_\epsilon^\prime(\Vert	u(s)-\psi\Vert_H^2)\langle G(u)dW(s), u-\psi\rangle+\displaystyle\int_0^tf_\epsilon^{\prime}(\Vert	u(s)-\psi\Vert_H^2)\Vert G(u)\Vert_{L_2(H_0,H)}^2 ds\notag\\
	&+2\displaystyle\int_0^tf_\epsilon^{\prime\prime}(\Vert	u(s)\Vert_H^2)\Vert( G(u),u-\psi)\Vert_{L_2(H_0,\mathbb{R})}^2 ds\label{ergodicity-est-2}.		\end{align}
	Since	$f^{\prime\prime}_\epsilon<0$,	then	the	last	term	is	non	positive.  In addition, recall	that	$f^\prime_\epsilon\leq	1$ and		the	stochastic	integral	is	a	$(\mathcal{F}_t)$-martingale.	On	the	other	hand,
	by	using	that	$f^\prime_\epsilon> 0$,			$H_{2,1}$,	$H_{3}$, \eqref{coerv-A-psi}, \eqref{est-F-psi} and \eqref{est-f-psi} we	obtain	for	any	$K>0$
    \begin{align}
&\E	f_\epsilon(\Vert u(t)-\psi \Vert_H^2)+\alpha\displaystyle\E	\int_0^tf_\epsilon^\prime(\Vert	u(s)-\psi\Vert_H^2)\Vert u(s)\Vert_V^pds\notag\\&\leq (l_1  +\bar K\Vert \psi\Vert_V+C(\bar{K},\alpha)\Vert \psi\Vert_V^p+C(\alpha,p)\Vert	f\Vert_{V^\prime}^{p^\prime}+\Vert	f\Vert_{V^\prime}\Vert		\psi\Vert_V)t+f_\epsilon(\Vert u_0 -\psi\Vert_H^2)\notag\\&+ 2 \lambda\E	\int_0^t f_\epsilon^\prime(\Vert	u(s)-\psi\Vert_H^2)\Vert u(s)\Vert_H^2ds+2\E	\int_0^t f_\epsilon^\prime(\Vert	u(s)-\psi\Vert_H^2)\vert( \mathbf{F}(u),u-\psi)\vert ds\notag\\
	&+ \frac{\alpha}{4} \displaystyle\E	\int_0^tf_\epsilon^\prime(\Vert	u(s)-\psi\Vert_H^2) \Vert u(s)\Vert_V^pds+M(1+K^2)t+	\dfrac{L_G(1+K^2)}{K^2}\displaystyle\E	\int_0^tf_\epsilon^\prime(\Vert	u(s)-\psi\Vert_H^2)\Vert	u(s)\Vert_H^2 ds.\notag\end{align}
	By using  \eqref{est-F-psi}, we have
	\begin{align*}
&\E	\int_0^t f_\epsilon^\prime(\Vert	u(s)-\psi\Vert_H^2)\vert( \mathbf{F}(u),u-\psi)\vert ds \\
&\leq L_F\E\int_0^t f_\epsilon^\prime(\Vert	u(s)-\psi\Vert_H^2)\Vert		u\Vert_H^2ds\\&+\frac{\alpha}{8}\E\int_0^t f_\epsilon^\prime(\Vert	u(s)-\psi\Vert_H^2)\Vert	u\Vert_{V}^p ds+ C(\alpha,p,\Vert	\mathbf{F}(0)\Vert_{H},\Vert\psi\Vert_V,L_F)t,
\end{align*}
thus, we get
\begin{align}
&\E	f_\epsilon(\Vert u(t)-\psi \Vert_H^2)+\dfrac{\alpha}{2}\displaystyle\E	\int_0^tf_\epsilon^\prime(\Vert	u(s)-\psi\Vert_H^2)\Vert u(s)\Vert_V^pds\notag\\&\leq (l_1  +\bar K\Vert \psi\Vert_V+C(\bar{K},\alpha)\Vert \psi\Vert_V^p+C(\alpha,p)\Vert	f\Vert_{V^\prime}^{p^\prime}+\Vert	f\Vert_{V^\prime}\Vert		\psi\Vert_V)t+f_\epsilon(\Vert u_0 -\psi\Vert_H^2)\notag\\&+ (2 \lambda+2L_F+\dfrac{L_G(1+K^2)}{K^2})\E	\int_0^t f_\epsilon^\prime(\Vert	u(s)-\psi\Vert_H^2)\Vert u(s)\Vert_H^2ds\notag\\
	&+ M(1+K^2)t+C(\alpha,p,\Vert	\mathbf{F}(0)\Vert_{H},\Vert\psi\Vert_V,L_F)t.\notag\end{align}
In conclusion, there exists $S>0$ depending only on the data such that 
  \begin{align}
&\E	f_\epsilon(\Vert u(t)-\psi \Vert_H^2)+\dfrac{\alpha}{2}\displaystyle\E	\int_0^tf_\epsilon^\prime(\Vert	u(s)-\psi\Vert_H^2)\Vert u(s)\Vert_V^pds\notag\\&\leq f_\epsilon(\Vert u_0 -\psi\Vert_H^2)+ (2 \lambda+2L_F+\dfrac{L_G(1+K^2)}{K^2})\E	\int_0^t f_\epsilon^\prime(\Vert	u(s)-\psi\Vert_H^2)\Vert u(s)\Vert_H^2ds+St.\notag\end{align}  
Similarly, if $G$ satisfies \eqref{bound-noise-p=2}, one obtains
  \begin{align}
&\E	f_\epsilon(\Vert u(t)-\psi \Vert_H^2)+\dfrac{\alpha}{2}\displaystyle\E	\int_0^tf_\epsilon^\prime(\Vert	u(s)-\psi\Vert_H^2)\Vert u(s)\Vert_V^pds\notag\\&\leq f_\epsilon(\Vert u_0 -\psi\Vert_H^2)+ (2 \lambda+2L_F)\E	\int_0^t f_\epsilon^\prime(\Vert	u(s)-\psi\Vert_H^2)\Vert u(s)\Vert_H^2ds+(S+\mathbf{K})t.\notag\end{align}
\end{itemize}
\end{proof}
%\section*{Acknowledgements}
%This work is funded by national funds through the FCT - Funda\c c\~ao para a Ci\^encia e a Tecnologia, I.P., under the scope of the projects UIDB/00297/2020 and UIDP/00297/2020 (Center for Mathematics and Applications).

\end{document}